\documentclass{amsart}
\usepackage{hyperref}
\usepackage{cite}
\usepackage{amssymb}
\usepackage{amsmath}
\usepackage{amsthm}
\usepackage{amsfonts}
\usepackage{graphicx}%
\usepackage{comment}
\usepackage{esint}

\usepackage{color}
\usepackage{bbm}

\usepackage[toc,page]{appendix}

\usepackage{stmaryrd}
\usepackage{longtable}

\newtheorem{theo}{Theorem}[section]
\newtheorem{pro}[theo]{Proposition}
\newtheorem{lem}[theo]{Lemma}
\newtheorem{cor}[theo]{Corollary}

\theoremstyle{definition}
\newtheorem{defin}[theo]{Definition}
\newtheorem{exa}[theo]{Example}

\newtheorem{notation}[theo]{Notation}

\theoremstyle{remark}
\newtheorem{rem}[theo]{Remark}
\numberwithin{equation}{section}

\newcommand{\SU}{\operatorname{SU}}

\newcommand{\tr}{\operatorname{tr}}

\newcounter{keepgoing}

\begin{document}

\newcommand{\sgn}{\operatorname{sgn}}

\newcommand{\pminterval}[1]{\Biggl[- {#1}, {#1}\Biggr]}

\title{Uniform doubling for abelian products with $\SU(2)$}

\author[Eldredge ]{Nathaniel Eldredge {$^{\dag}$}}
\thanks{\footnotemark {$^{\dag }$} Research was supported in part by a grant from the Simons Foundation (\#355659, Nathaniel Eldredge), and by funding from the University of Northern Colorado's Faculty Research and Publications Board.}
\address{$^{\dag }$ Department of Mathematical Sciences\\
  University of Northern Colorado\\
  Greeley, CO 80639, U.S.A.
}
\email{neldredge@unco.edu}

\author[Gordina]{Maria Gordina{$^{\dag \dag}$}}
\thanks{\footnotemark {$\dag \dag$} Research was supported in part by NSF Grant DMS-2246549.}
\address{$^{\dag \dag}$ Department of Mathematics\\
University of Connecticut\\
Storrs, CT 06269,  U.S.A.}
\email{maria.gordina@uconn.edu}

\author[Saloff-Coste]{Laurent Saloff-Coste{$^{\ddag }$}}
\thanks{\footnotemark {$\ddag $} Research was supported in part by NSF grants DMS-2054593 and DMS-2343868 and by a Simons Fellowship \#915128.}
\address{$^{\ddag }$
Department of Mathematics \\
Cornell University
}
\email{lsc@math.cornell.edu}



\begin{abstract} We prove that the uniform doubling property holds for every Lie group  which can be written as a quotient group of $\operatorname{SU}(2) \times \mathbb{R}^n$ for some $n$. In particular, this class includes the four-dimensional unitary group $\operatorname{U}(2)$. As this class contain non-compact as well as compact Lie groups, we discuss a number of analytic and spectral consequences
for the corresponding heat kernels.
\end{abstract}

\maketitle 

\tableofcontents

\section{Introduction}\label{s.intro}

The purpose of this paper is to extend the study of the phenomenon of \textbf{uniform volume doubling} on Lie groups, begun by the authors in \cite{EldredgeGordinaSaloff-Coste2018}.

Let $G$ be a finite-dimensional connected real Lie group with identity $e$; for simplicity, assume $G$ is unimodular.  Given a left-invariant Riemannian metric $g$ on $G$, denote by $B_g(x,r)$ the ball with respect to the induced Riemannian distance $d_g$, centered at $x \in G$ with radius $r \geqslant 0$.  Let $\mu_g$ denote the Riemannian volume measure induced by $g$, which is some rescaling of the Haar measure on $G$.  The \textbf{volume doubling constant} of $(G,g)$ is defined as
\begin{equation}
 D_g = \sup_{x \in G, r > 0} \frac{\mu_g(B_g(x,2r))}{\mu_g(B_g(x,r))} = \sup_{r > 0} \frac{\mu_g(B_g(e,2r))}{\mu_g(B_g(e,r))}
\end{equation}
where the ratio $\frac{\mu_g(B_g(x,2r))}{\mu_g(B_g(x,r))}$ is independent of $x \in G$ by the left invariance.

Now let $\mathfrak{L}(G)$ denote the set of \emph{all} left-invariant Riemannian metrics on $G$.   If there is a finite constant $D(G)$, depending only on the group $G$, such that $D_g \leqslant D(G)$ for all $g \in \mathfrak{L}(G)$, we say that $G$ is \textbf{uniformly doubling}.

A trivial example of a uniformly doubling Lie group is the abelian Lie group $\mathbb{R}^n$, for which we have $D(\mathbb{R}^n) = 2^n$.  It follows from this that every abelian Lie group, being a quotient of some $\mathbb{R}^n$, is uniformly doubling as well; see Section~\ref{quotients} for more on the behavior of volume doubling under quotients.  Beyond these trivial cases, it is not obvious that any other examples exist.  However, we showed in \cite{EldredgeGordinaSaloff-Coste2018} that the three-dimensional special unitary group $\mathrm{SU}(2)$ is uniformly doubling.  Our proof was quite specific to the structure of $\mathrm{SU}(2)$ itself, and does not directly generalize to other Lie groups.   Nonetheless, there are heuristic reasons to believe that the uniform doubling property is actually satisfied by \emph{every} compact connected Lie group, which we stated as a conjecture in \cite{EldredgeGordinaSaloff-Coste2018}.  However, as we currently lack general techniques for attacking the full conjecture, our more immediate goal is to find additional specific examples of uniformly doubling Lie groups.

After $\operatorname{SU}(2)$, a natural next target was the four-dimensional unitary group $\mathrm{U}(2)$.  In this paper, we prove that $\mathrm{U}(2)$ is indeed uniformly doubling.  In fact, we prove the uniform doubling property holds for every Lie group in a larger class: namely, those which can be written as a quotient group of $\mathrm{SU}(2) \times \mathbb{R}^n$ for some $n$.    Note that $\mathrm{U}(2)$ itself is indeed of this form, as we have a surjective homomorphism from $\mathrm{SU}(2) \times \mathbb{R}^1$ onto $\mathrm{U}(2)$ defined by $(A, t) \mapsto e^{it} A$, whose kernel is the normal subgroup generated by $(-I, \pi)$.
  
\begin{theo}\label{main-intro}
Let $G$ be a Lie group such that $G \cong (\mathrm{SU}(2) \times \mathbb{R}^n) / H$ for some closed normal subgroup $H \leqslant \mathrm{SU}(2) \times \mathbb{R}^n$.  Then $G$ is uniformly doubling, with a constant $D(G)$ that depends only on $n$.  

In particular, the unitary group $\mathrm{U}(2)$ is uniformly doubling, as are all groups of the form $\mathrm{SU}(2) \times A$, $\mathrm{SO}(3) \times A$, $\mathrm{U}(2) \times A$, where $A$ is any connected finite-dimensional abelian Lie group. 
\end{theo}

One could give an alternative phrasing of Theorem \ref{main-intro} by noting that $G$ is of the form $G \cong (\mathrm{SU}(2) \times \mathbb{R}^n) / H$ if and only if the universal cover of $G$ is $\mathrm{SU}(2) \times \mathbb{R}^m$ (or $\mathbb{R}^m$, in the trivial case when $G$ is abelian) for some $m \le n$.
  
Our results also imply a certain uniform doubling property for homogeneous spaces over $\mathrm{SU}(2) \times \mathbb{R}^n$, which we can view as quotients of $\mathrm{SU}(2) \times \mathbb{R}^n$ by non-normal subgroups.  However, we concentrate on the case of quotient Lie groups.  See Corollary \ref{uniformly-homogeneous} and its accompanying discussion. 

We remark that although $\mathrm{SU}(2)$ and $\mathbb{R}^n$ are each uniformly doubling (the former as proved in \cite{EldredgeGordinaSaloff-Coste2018}, the latter trivially), it does not follow trivially that the product Lie group $\mathrm{SU}(2) \times \mathbb{R}^n$ is uniformly doubling.  In general, if $G_1,G_2$ are uniformly doubling Lie groups, what does follow trivially is that the class of all \emph{product} left-invariant Riemannian metrics on $G = G_1 \times G_2$ is uniformly doubling.  That is, if we consider the class of Riemannian metrics $g$ on $G$ that are of the form $g = g_1 \oplus g_2$ for $g_i \in \mathfrak{L}(G_i)$, then there is a uniform upper bound on the doubling constants of all such metrics $g$, given indeed by $(D(G_1) \cdot D(G_2))^2$.  However, this is not the same as saying that $G$ is uniformly doubling, because the latter requires a uniform upper bound on the doubling constants of \emph{all} left-invariant Riemannian metrics $g$ on $G$, most of which are \emph{not} of the form $g = g_1 \oplus g_2$.  See Proposition~\ref{doubling-product} and Remark~\ref{non-product} below for further discussion of this issue.

In \cite{EldredgeGordinaSaloff-Coste2018}, our starting point was a result of Milnor that each left-invariant Riemannian metric on $\mathrm{SU}(2)$ (or, equivalently, every inner product on its Lie algebra $\mathfrak{su}(2)$) admits an orthogonal basis of a special form, which we refer to as a \emph{standard Milnor basis}.  Using this, we obtained a parametrization of $\mathfrak{L}(\mathrm{SU}(2))$.  However, because of the need to consider non-product metrics as noted above, this does not immediately yield a similar parametrization of the left-invariant Riemannian metrics on $\mathrm{SU}(2) \times \mathbb{R}^n$. Nonetheless, we do eventually obtain a parametrization of a certain sub-class of these metrics, in terms of a similar special basis, and this turns out to be sufficient for our purposes.  This is carried out in Section \ref{decoupled}.

Volume doubling is widely used as an essential tool in geometric analysis, particularly in the context of metric measure spaces, where it can often substitute for curvature bounds.  Indeed, the celebrated Bishop--Gromov comparison theorem shows that Ricci curvature lower bounds for a Riemannian metric imply an upper bound on its volume doubling constant.  (Though not conversely; in fact, a uniformly doubling Lie group will typically \emph{not} admit a uniform lower bound on the Ricci curvatures of its left-invariant Riemannian metrics.  See \cite[Section~1.2]{EldredgeGordinaSaloff-Coste2018} for further discussion and examples.)  As such, the uniform doubling property should be of inherent interest in its own right.

However, in the setting of Lie groups in particular, volume doubling is especially interesting, because it implies a number of other important functional inequalities.  The starting point is a scale-invariant Poincar\'e inequality of the form
  \begin{equation}\label{poincare-intro}
    \int_{B_g(x,r)} |f - f_{x,r}|^2 \,d\mu_g \leqslant 2 r^2 D_g \int_{B_g(x, 2r)} |\nabla_g f|_g^2\,d\mu_g
  \end{equation}
where $f_{x,r} = \fint_{B_g(x,r)} f\,d\mu_g$ is the mean of $f$ over $B_g(x,r)$.  Note that the constant in \eqref{poincare-intro} is precisely the volume doubling constant of the left-invariant Riemannian metric $g$, and so when $G$ is uniformly doubling, we have \eqref{poincare-intro} with the same constant $D(G)$ for all $g \in \mathfrak{L}(G)$; that is, we have a uniform Poincar\'e inequality.  The Poincar\'e inequality in turn can be used to prove a number of other functional inequalities, such as heat kernel estimates, and, in the case of a compact group $G$, spectral gap and Weyl eigenvalue counting estimates.  We briefly review some of these consequences in Section~\ref{s.consequences} below, but refer to \cite[Section~8]{EldredgeGordinaSaloff-Coste2018} for a more extended development and survey of related literature.

The structure of the paper is as follows.  In Section \ref{quotients}, we discuss, at a general level, how the volume doubling property passes from a Lie group to its quotient groups and homogeneous spaces.  This leaves us only needing to prove Theorem \ref{main-intro} for $\mathrm{SU}(2) \times \mathbb{R}^n$ itself.  In Section \ref{decoupled}, we perform a further reduction, showing that it suffices to handle the case $n=3$, and moreover (Theorem \ref{doubling-reduction}) that we can restrict our attention to a smaller class $\mathfrak{L}_{\mathrm{dec}}(\mathrm{SU}(2) \times \mathbb{R}^3)$ of left-invariant Riemannian metrics on $\mathrm{SU}(2) \times \mathbb{R}^3$, which we call \emph{decoupled}.  Such a metric admits an orthogonal basis satisfying useful algebraic properties, analogous to (but more complicated than) those of a standard Milnor basis which we used for $\mathrm{SU}(2)$ in \cite{EldredgeGordinaSaloff-Coste2018}, and we get a convenient parametrization of this restricted space of metrics $\mathfrak{L}_{\mathrm{dec}}(\mathrm{SU}(2) \times \mathbb{R}^3)$.  For decoupled metrics $g$, we are able to produce an explicit uniform two-sided estimate on the volumes of balls $B_g(x,r)$ (Theorem \ref{volume-estimate}).  The short Section \ref{volume-statement} is devoted to stating this estimate, with accompanying notation, and explaining how it implies uniform doubling for $\mathfrak{L}_{\mathrm{dec}}(\mathrm{SU}(2) \times \mathbb{R}^3)$ (Corollary \ref{decoupled-uniformly-doubling}. Then Corollary \ref{decoupled-uniformly-doubling}, together with the reduction from Theorem \ref{doubling-reduction}, implies the main result,  Theorem~\ref{main-intro}.

In the remainder of the paper, Sections~\ref{identities-sec}--\ref{upper-sec}, we carry out the proof of Theorem \ref{volume-estimate}.  Section~\ref{identities-sec} contains computations of various identities involving the exponential map in $\mathrm{SU}(2)$, which are useful in estimating the distance to a point; they are similar in spirit to those in \cite[Section~2.4]{EldredgeGordinaSaloff-Coste2018} but more detailed.  Section~\ref{coordinates-sec} introduces a convenient coordinate system on $\mathrm{SU}(2) \times \mathbb{R}^3$ and defines explicit subsets of the coordinate chart whose images turn out to be comparable to the metric balls.  Finally, Sections \ref{lower-sec} and \ref{upper-sec} respectively prove the lower and upper bounds asserted by Theorem~\ref{volume-estimate}, by showing that a ball contains (respectively, is contained in) a set of the form described in Section~\ref{coordinates-sec}, whose volume can be computed by elementary means.

\section{Uniform doubling for quotients and homogeneous spaces}\label{quotients}

An important fact about volume doubling is that it is preserved \emph{quantitatively} by passage to a quotient.  That is, if $G, H$ are two Lie groups with a surjective homomorphism $\pi : G \to H$, and
$g,h$ are left-invariant Riemannian (or sub-Riemannian) metrics on $G,H$ respectively such that $\pi$ is a submersion, then the doubling constant of $(H,h)$ can be bounded in terms of the doubling constant of $(G,g)$.

In this paper, we take advantage of this fact twice.  First, we show that for any group of the form $H = \operatorname{SU}(2) \times \mathbb{R}^n$, there is a larger group $G$ (namely $G = \operatorname{SU}(2)
\times \mathbb{R}^{n+3}$) such that every left-invariant metric on $H$ can be lifted to a left-invariant metric on $G$ which is of a certain convenient form that we call \emph{decoupled}.  Thus it is sufficient to show that the class $\mathcal{L}_{\operatorname{dec}}(G)$ of \emph{decoupled} left-invariant metrics on $G$ is uniformly doubling, which is convenient because the decoupled property greatly simplifies
the explicit computations needed for volume estimates.  Once this is done, by applying the quotient fact a second time, we get as an immediate corollary the uniform doubling for any group $K$ that is a quotient of some $\operatorname{SU}(2) \times \mathbb{R}^n$, such as for instance $K = \operatorname{U}(2)$.

The general principle that doubling is preserved under quotients seems to be well-known to experts, but the details are somewhat scattered around the literature.  As such, to make this paper more self-contained, we will review the argument here. We begin in the more general setting of homogeneous spaces, although this generality is not needed for the rest of the paper.  The special case of quotient Lie groups is discussed at the end of the section.

Let $G$ be a unimodular connected Lie group with a smooth transitive right action on a smooth manifold $M$.  Fix $p \in M$ and define $\pi: G \to M$ by $\pi(\sigma) = p \cdot \sigma$.  Let $S_{p} \subset G$ be the stabilizer of $p$, so that $S_{p}$ is a closed subgroup of $G$.  Also,
let $\mu$ be a bi-invariant Haar measure on $G$.

The following lemma is essentially a special case of \cite[Proposition 2.28]{GallotHulinLafontaineBook3rdEdition}. 

\begin{lem}\label{pushforward-metric}
For each left-invariant Riemannian metric $g$ on $G$, there is a unique Riemannian metric $h$ on $M$, which we denote as the push-forward $h = \pi_{\ast} g$, for which $\pi$ is a Riemannian  submersion.  Moreover, $h$ does not depend on the choice of $p$.
\end{lem}

\begin{proof}
Note that $S_{p}$ acts as a group of isometries on $G$ by  multiplication, so that $M$ is diffeomorphic to the right coset space $G \backslash S_{p}$, and the action of $G$ on this space is indeed on
the right.  Thus we may apply \cite[Proposition 2.28]{GallotHulinLafontaineBook3rdEdition}, taking $(\tilde{M}, \tilde{g}) = (G, g)$ and $G=S_{p}$, to obtain the existence and uniqueness of the Riemannian metric $h$ (which is $g$ in their notation) for which $\pi$ is a Riemannian submersion.  The proof is
straightforward: for each $\sigma \in G$, there is only one possible way to define $h$ on $T_{\pi(\sigma)} M$, and it is then routine to verify that it is well-defined and smooth on $M$.

Independence of the choice of $p$ also follows from the left invariance. Let $h = \pi_{\ast} g$ be as above.  Take some other $p^{\prime} \in M$, and let $\pi^{\prime}(\sigma) = p^{\prime} \cdot \sigma$.  Then $\pi^{\prime}_{\ast} g$ is the unique Riemannian metric on $M$ for which $d\pi^{\prime}$ is a Riemannian submersion. We now claim that $d\pi^{\prime}$ is also a Riemannian submersion for the metric $h = \pi_{\ast} g$, from which it follows that $h = \pi^{\prime}_{\ast} g$. Indeed, since the action of $G$ is transitive, we can find $\tau \in G$ with $p^{\prime} = p \tau$, so $\pi^{\prime} = \pi \circ L_\tau$.  Thus $d\pi^{\prime} = d\pi \circ dL_{\tau}$, where $d\pi$ is a partial isometry from $(G, g)$ to $(M, h)$, and $dL_{\tau}$ is an isometry of $(G,g)$.  So $d\pi^{\prime}$ is
also a partial isometry from $(G,g)$ to $(M,h)$, which is to say that $\pi^{\prime}$ is a Riemannian submersion.
\end{proof}

Note that the Riemannian metric $h = \pi_{\ast} g$ is not necessarily $G$-invariant, unless the metric $g$ is bi-invariant.  The following example illustrates this phenomenon.

\begin{exa}\label{sphere-non-invariant}
Let $M = S^2$ be the unit sphere of $\mathbb{R}^3$, realized as a set of row vectors, with $p = (0,0,1)$ the north pole.  Consider the right action of $G =\operatorname{SO}(3)$ on $S^2$ by matrix multiplication on the right.  A basis for $T_e \operatorname{SO}(3) =\mathfrak{so}(3)$ is given by the matrices
\begin{equation*}
 u_1 =
 \begin{pmatrix}
   0 & 0 & 0 \\ 0 & 0 & -1 \\ 0 & 1 & 0
 \end{pmatrix},
 u_2 =
 \begin{pmatrix}
   0 & 0 & -1 \\ 0 & 0 & 0 \\ 1 &0 & 0
 \end{pmatrix},
 u_3 =
 \begin{pmatrix}
   0 & -1 & 0 \\ 1 & 0 & 0 \\ 0 & 0 & 0
 \end{pmatrix}.
\end{equation*}
It is easily checked that
\begin{equation*}
 d\pi_e u_1 = \partial_y, \qquad d\pi_e u_2 = \partial_x, \qquad
 d\pi_e u_3 = 0.
\end{equation*}
Let $g$ be a left-invariant Riemannian metric on $\operatorname{SO}(3)$ for which $u_1, u_2, u_3$ are orthogonal with distinct lengths $g_e(u_i, u_i) = a_i^2$, $i=1,2,3$.  Then $\operatorname{Hor}_e$ is spanned by $u_1, u_2$, so the induced Riemannian metric $h$ on $S^2$ makes $\partial_x, \partial_y \in T_p S^2$ orthogonal, with $h_p(\partial_x, \partial_x) = a_2^2$, $h_p(\partial_y, \partial_y) = a_1^2$.

Now let 
$\sigma = \left(
  \begin{smallmatrix}
 0 & 0  & 1 \\ 0 & 1 & 0 \\ -1 & 0 & 0
\end{smallmatrix} \right)$ 
act as a $\pi/2$ counterclockwise rotation about the $y$ axis, so that $q = p \cdot \sigma = (-1, 0, 0)$ is the \emph{west pole}.  One can compute that
\begin{equation*}
d\pi_\sigma dL_\sigma e_1 = 0, \qquad d\pi_\sigma dL_\sigma e_2 =
\partial_z, \qquad d\pi_\sigma dL_\sigma e_3 = \partial_y
\end{equation*}
and by the left-invariance of $g$, we have that $dL_\sigma e_i, i=1, 2, 3$ are orthogonal in $T_\sigma \operatorname{SO}(3)$, so $\operatorname{Hor}_\sigma$ is spanned by $e_2, e_3$.  Thus $\partial_y, \partial_z$ are orthogonal in $T_q S^2$,   with lengths $h(\partial_y, \partial_y) = a_3^2$, $h(\partial_z, \partial_z) = a_2^2$.

On the other hand, if $A_\sigma(x) = x \sigma$ is the action of $\sigma$ on $S^2$, we have
\begin{equation*}
dA_\sigma \partial_x = \partial_z, \qquad dA_\sigma \partial_y = \partial_y.
\end{equation*}
Thus $h_\sigma(dA_\sigma \partial_y, dA_\sigma \partial_y) =
  h_\sigma(\partial_y, \partial_y) = a_3^2 \ne a_1^2 = h_p(\partial_y,
  \partial_y)$, so that $h$ is not $G$-invariant.

Moreover, the Riemannian volume measure $\mu_h$ of $h$ is not $G$-invariant either.  One may compute that $|\det dA_\sigma| = |a_3/a_1| \not= 1$,  so that $A_\sigma$ is not volume preserving.  In particular, $\mu_h$ does not coincide up to scaling with the  invariant measure $\nu$ (i.e., Lebesgue surface area measure on the round sphere),   and so the volume doubling property for the metric measure space $(S^2, d_h, \nu)$ is not necessarily equivalent to that of $(S^2, d_h, \mu_h)$.  Moreover, one can verify that the densities $d\mu_h/d\nu$ have no uniform upper or lower bound as  the Riemannian metric $h$ varies over the class of metrics $h =  \pi_{\ast} g$, $g \in \mathcal{L}(\operatorname{SO}(3))$, so the uniform doubling of the class of metric measure spaces $(S^2, d_h, \mu_h)$ is not a priori equivalent to that of the class of $(S^2, d_h, \nu)$.
\end{exa}

\begin{rem}\label{hopf-non-invariant}
  One can also construct an example similar to Example \ref{sphere-non-invariant} in terms of the Hopf fibration, replacing $\mathrm{SO}(3)$ with its double cover $\mathrm{SU}(2)$ and working with the Hopf map $\pi : \mathrm{SU}(2) \to S^2$.
\end{rem}

It is elementary to show that every Riemannian submersion is a \emph{submetry} of the Riemannian manifolds viewed as metric spaces, which means that it maps each closed ball to a closed ball of the same radius; see for example \cite[Exercise 23, p.~151]{PetersenBook2ndEdition}.  Thus if we write $B(\cdot, r)$ for the closed balls of $(G,g)$, $(M,h)$ as appropriate, we have $\pi(B(\sigma, r)) = B(\pi(\sigma), r)$.

Suppose now that $M$ admits a $G$-invariant Radon measure $\nu$, which as noted above, need not agree with the Riemannian volume measure for the metric $h$.  We now study the relationship between the volumes of those balls with respect to the measures $\mu$ and $\nu$.

Our main tool is the following useful lemma taken from \cite[Lemme 1.1]{Guivarch1973a}, slightly adapted.
For the reader's convenience, we include the proof here. Note that the statement in \cite[Lemme 1.1]{Guivarch1973a} assumes a left action, whereas here we have a right action, and also includes a
set $Y$ which we take to be the singleton $\{p\}$.

\begin{lem}[ {\cite[Lemme 1.1]{Guivarch1973a}}]\label{guivarch}
Let $A, B$ be two compact subsets of $G$.  Then 
\begin{equation}\label{guivarch-eqn}
 \mu(A) \nu(\pi(B)) \leqslant \mu(AB) \nu(\pi(A^{-1})).
 \end{equation}
\end{lem}

\begin{proof}
Consider the \emph{parallelogram} set $P \subset G \times M$ defined by
\begin{equation*}
 P = \{(\sigma, q): \sigma \in AB,  q \cdot \sigma^{-1} \in \pi(A^{-1})\}.
\end{equation*}
Note that for each $\sigma \in AB$, the vertical cross-section of $P$ at $\sigma$ is $P_\sigma = \pi(A^{-1}) \cdot \sigma$.  By the $G$-invariance of $\nu$, we have $\nu(P_\sigma) = \nu(\pi(A^{-1}))$. So by Fubini's theorem, the product measure of $P$ is given by
\begin{equation}\label{guiv1}
(\mu \times \nu)(P) = \mu(AB) \nu(\pi(A^{-1})).
\end{equation}

On the other hand, suppose $q \in \pi(B)$, so that there exists $\beta  \in B$ with $q = \pi(\beta)$. If $\sigma \in A \beta \subset AB$, so that $\sigma = \alpha \beta$ for some $\alpha \in A$, then
\begin{equation*}
q \cdot \sigma^{-1} = \pi(\beta \sigma^{-1}) = \pi(\alpha^{-1}) \in \pi(A^{-1}),
\end{equation*}
so that $(\sigma, q) \in P$.  Thus for $q \in \pi(B)$, the horizontal cross-section $P^q$ contains $A\beta$, so that $\mu(P^q) \geqslant \mu(A\beta) = \mu(A)$ by right invariance of $\mu$.  Using Fubini's theorem again, we have
\begin{equation}\label{guiv2}
 (\mu \times \nu)(P) \geqslant \mu(A) \nu(\pi(B)),
\end{equation}
and comparing \eqref{guiv1} with \eqref{guiv2} yields the conclusion.
\end{proof}

\begin{theo}\label{doubling-quotient}
The metric measure spaces $(M, d_h, \nu)$ and $(G, d_g, \mu)$ have volume doubling constants related by
\begin{equation}
 D(M, d_h, \nu) \leqslant D(G, d_g, \mu)^2.
\end{equation}
\end{theo}

\begin{proof}
We apply Lemma~\ref{guivarch} with $A = B(e, r)$, $B =  B(e, 2r)$.  Since the distance $d_g$ is left-invariant, we have  $A^{-1} = A$ and $AB = B(e, 3r)$. And since $\pi$ is a submetry, we have $\pi(A^{-1}) = \pi(A) = B(p,  r)$, $\pi(B) = B(p, 2r)$.  So \eqref{guivarch-eqn} reads
  \begin{equation*}
 \mu(B(e, r)) \nu(B(p, 2r)) \leqslant \mu(B(e, 3r)) \nu(B(p, r))
  \end{equation*}
  which rearranges to give
  \begin{equation*}
 \frac{\nu(B(p,2r))}{\nu(B(p,r))} \leqslant \frac{\mu(B(e,
   3r))}{\mu(B(e, r))} \leqslant D(G, d_g, \mu)^2
  \end{equation*}
  since $3r \leqslant 4r$.  Taking the supremum over $r$ and noting that $p
  \in M$ was arbitrary, we have the result.
\end{proof}

\begin{cor}\label{uniformly-homogeneous}
If $G$ is uniformly doubling, then so is the family of metric  measure spaces $\left\{(M, d_{\pi_{\ast} g}, \nu), g \text{ is a left-invariant metric on } G \right\}$.
\end{cor}

We emphasize that since $\nu$ is not necessarily the volume measure of $h$ (even up to a constant), the metric measure space $(M, d_h, \nu)$ in general is not the metric measure space structure associated with the Riemannian structure on $(M, h)$.  However, we now specialize to the case of a quotient group, where the situation is nicer.  Let $S$ be a (closed) normal subgroup  of $G$ (unless otherwise stated, we assume subgroups are closed), and let $M = H = G / S$ be the quotient group. We will use $e$ to denote either the identity element of $G$ or of $H$, where it should be clear from context which is meant.  We apply the above results to the right action of $G$ on $H$ by multiplication, with $\pi : G \to H$ the quotient map, so that the point $p$ in the previous development is taken as the identity $e$ of $H$.

\begin{lem}\label{quotient-lift}
Given a left-invariant Riemannian metric $g$ on $G$, define a Riemannian metric $h = \pi_{\ast} g$ on $H$ as in Lemma \ref{pushforward-metric}. Then $h$ is left-invariant.  Moreover, every left-invariant Riemannian metric on $H$ arises in this way.
\end{lem}

\begin{proof}
Let $q = \pi(\sigma) \in H$.  Suppose $v \in T_e H$; we will show that $h_q(dL_q v, dL_q v) = h_e(v,v)$.  Let $w \in \operatorname{Hor}_e$ be the unique horizontal tangent vector in $T_e G$ with $d\pi_e w = v$.  We claim that $dL_\sigma w \in \operatorname{Hor}_\sigma$.  Indeed, suppose that $u \in \ker d\pi_\sigma$.  Since $\pi$ is a group homomorphism, we have $\pi \circ L_{\sigma^{-1}} = L_{q^{-1}} \circ \pi$ and thus
\begin{equation*}
d \pi_e dL_{\sigma^{-1}} u = dL_{q^{-1}} d\pi_{\sigma} u = 0,
\end{equation*}
so $dL_{\sigma^{-1}} u \in \ker d\pi_e$.  As such,
\begin{equation*}
g_\sigma(dL_\sigma w, u) = g_e(w, dL_{\sigma^{-1}} u) = 0
\end{equation*}
and so $dL_\sigma w \in \operatorname{Hor}_\sigma$.  Moreover, $d\pi_\sigma dL_\sigma w = dL_q d\pi_e w = dL_q v$.  Now since $\pi$ is a Riemannian submersion, we have
\begin{equation*}
h_q(dL_q v, dL_q v) = g_\sigma(dL_\sigma w, dL_\sigma w) = g_e(w,
 w) = h_e(v, v).
\end{equation*}

Now suppose we are given a left-invariant Riemannian metric $h$ on $H$.  Since $d\pi_e$ is surjective, we can choose an inner product $g_e$ on $T_e G$ for which $d\pi_e$ is a partial isometry. This choice is unique if and only if $\dim H = \dim G$. Extend this inner product $g_e$ to a left-invariant Riemannian metric on $G$, i.e. by setting $g_\sigma(w, w) := g_e(dL_{\sigma^{-1}} w, dL_{\sigma^{-1}} w)$. The induced Riemannian metric $\pi_{\ast} g$ on $H$  is left-invariant as noted above, and agrees with $h_e$ on $T_e H$ by construction, so $h = \pi_{\ast} g$.
\end{proof}

In this setting, the Riemannian volume measure for the left-invariant metric $h$ is (up to scaling) the left Haar measure $m_L$ on $H$, whereas the $G$-invariant measure $\nu$ is the right Haar measure $m_R$.  So the foregoing results would describe the volume doubling property for the metric measure space $(H, d_h, m_R)$.  However, when $d_h$ is left (or right) invariant, the doubling constants with respect to $m_R$ and $m_L$ coincide.  Since $m_L$ and $d_H$ are both left-invariant, it suffices to consider the doubling constant of $m_L$ at the identity; but the invariance of the distance implies that each ball $B(e, r)$ is symmetric, that is, $B(e, r)^{-1} =B(e,r)$ and so $m_L(B(e,r)) = m_L(B(e,r)^{-1}) = m_R(B(e,r))$.  In any case, for the purposes of the present paper, we only need to consider $H$ that are quotients of the semisimple Lie group $\operatorname{SU}(2) \times \mathbb{R}^n$, which are themselves semisimple and hence unimodular, so the issue does not arise.

Thus, for quotient groups we can reformulate  Corollary \ref{uniformly-homogeneous} as follows.

\begin{cor}\label{uniformly-quotient}
Suppose $\mathcal{L}'(G) \subset \mathcal{L}(G)$,
  $\mathcal{L}'(H) \subset \mathcal{L}(H)$ are sets of
  left-invariant Riemannian metrics on $G,H$ respectively, such that
  for every $h \in \mathcal{L}'(H)$ there exists $g \in
  \mathcal{L}'(G)$ with $h = \pi_{\ast} g$.  Then if $\mathcal{L}'(G)$ is
  uniformly doubling, so is $\mathcal{L}'(H)$.
\end{cor}

In particular, Lemma \ref{quotient-lift} shows the hypothesis holds when $\mathcal{L}^{\prime}(G) = \mathcal{L}(G)$ and $\mathcal{L}^{\prime}(H) = \mathcal{L}(H)$, and therefore we have the following statement. 

\begin{cor}\label{uniformly-quotient-all}
If $G$ is uniformly doubling, then so is $H$.
\end{cor}

To conclude this section, we recall an elementary fact about volume doubling for product metrics.

\begin{pro}\label{doubling-product}
Let $(M_1, g_1), (M_2, g_2)$ be Riemannian manifolds, equipped with  the corresponding Riemannian volume measures $\mu_1, \mu_2$, and having volume doubling constants $D_1, D_2$ respectively.  Let $(M,g) = (M_1 \times M_2, g_1 \oplus g_2)$ be the product manifold equipped with the Riemannian product metric, and with $\mu$ its Riemannian volume  measure, which is the product measure $\mu = \mu_1 \otimes \mu_2$.  Then the volume doubling constant $D$ of $(M,g)$ is  bounded above by $D_1^2 D_2^2$.
\end{pro}

\begin{proof}
In this setting, it is easy to show that the Riemannian distances $d_1, d_2, d$ obey the Pythagorean identity $d((p_1, p_2), (q_1, q_2))^2 = d_1(p_1, q_1)^2 + d_2(p_2, q_2)^2$.  From this it follows that for any point $p = (p_1, p_2) \in M$, we have for the corresponding balls $B_1, B_2, B$ that
\begin{equation*}
 B_1(p_1, r/\sqrt{2}) \times B_2(p_2, r/\sqrt{2}) \subset B(p, r)
 \subset B(p, 2r) \subset B_1(p_1, 2r) \times B_2(p_2, 2r).
 \end{equation*}
Taking the volumes and rearranging, we have
  \begin{equation*}
 \frac{\mu(B(p,2r))}{\mu(B(p,r))} \leqslant \frac{\mu_1(B_1(p_1,
   2r))}{\mu_1(B_1(p_1, r/\sqrt{2}))} \cdot
   \frac{\mu_2(B_2(p_2,
   2r))}{\mu_2(B_2(p_2, r/\sqrt{2}))} \leqslant D_1^2 D_2^2.
  \end{equation*}
\end{proof}

\begin{rem}\label{non-product}
It is important to note that the above argument applies \emph{only} to the product metric $g = g_1 \oplus g_2$.  As such, in the setting of Lie groups, if $G_1, G_2$ are each uniformly doubling, it does not automatically follow that $G = G_1 \times G_2$ is uniformly doubling, since $G$ admits left-invariant metrics that are not product metrics.  It is true that given $g \in \mathcal{L}(G)$ and letting
$g_1, g_2$ be its restrictions to $G_1, G_2$ respectively, there is a minimum angle $\theta_g >0$ between
the tangent spaces of the two factors, given by
\begin{equation*}
  \cos \theta_g = \max\{\langle v_1, v_2 \rangle_{g} : v_i \in T_{p_i} G_i, g_i(v_i, v_i)=1\}
\end{equation*}
which by left-invariance is independent of the point $(p_1, p_2) \in G$.  We then have the bound $d_i(p_i, q_i) \leqslant \frac{d((p_1, p_2), (q_1, q_2))}{\sin \theta_g}$, and from this and the triangle inequality we obtain
\begin{equation*}
B_1(r/2) \times B_2(r/2) \subset B(r) \subset B(2r) \subset B_1\left(\frac{2r}{\sin\theta_g}\right)
 \times B_2\left(\frac{2r}{\sin\theta_g}\right)
\end{equation*}
yielding
\begin{equation*}
D_g \leqslant (D_1 D_2)^{\lceil 2 -\log_2(\sin \theta_g)\rceil}.
\end{equation*}
But since $\theta_g$ may be arbitrarily close to 0, this argument cannot yield a finite uniform bound on the doubling constant $D_g$ over all left-invariant metrics $g$ on $G$.
\end{rem}

\section{Decoupled metrics}\label{decoupled}

In this section we perform a series of reductions that leaves us needing only to prove uniform doubling for the specific group $\operatorname{SU}(2) \times \mathbb{R}^3$.   Moreover, we show that we need not even consider all possible left-invariant Riemannian metrics on this group, but can restrict consideration to a smaller class of \emph{decoupled} metrics, as defined in Definition~\ref{decoupled-def}, whose structure constants are of a particularly simple form.  Theorem~\ref{doubling-reduction} gives the precise statement of the reduction. As in \cite{EldredgeGordinaSaloff-Coste2018}, the starting point is the notion of a \emph{standard Milnor basis} for the Lie algebra $\mathfrak{su}(2)$, originated in \cite{Milnor1976}. We follow similar, but not identical, terminology and notation to
\cite[Sections 2.2--2.3]{EldredgeGordinaSaloff-Coste2018}.

Recall that by the classical paper \cite{Milnor1976} left-invariant metrics on $G$ are in one-to-one correspondence of inner products on the Lie algebra $\mathfrak{g}$.

\begin{notation}\label{n.LeftInvMetric}
For a left-invariant metric $g$ on $G$ we denote by  $\langle u, v \rangle_{g}$ the corresponding inner product on the Lie algebra $\mathfrak{g}$. We sometimes abuse notation and use $g$ instead of $\langle \cdot, \cdot \rangle_{g}$
\end{notation}

\subsection{Standard Milnor bases and generalizations}

\begin{defin}\label{d.MilnorBasis}
We say an ordered basis $( u_1, u_2, u_3 )$ for $\mathfrak{su}(2)$ is a \emph{standard Milnor basis} if it satisfies
\begin{equation}\label{e.BasisLie Brackets}
[u_1, u_2] = u_3, \qquad [u_2, u_3]=u_1, \qquad [u_3, u_1] = u_2.
\end{equation}
\end{defin}

\begin{exa}\label{pauli-ex}
The Pauli matrices
\begin{equation}
  \widehat{u}_1 = \frac{1}{2}
  \begin{pmatrix}
 0 & -i \\ -i & 0
  \end{pmatrix}, \quad
  \widehat{u}_2 = \frac{1}{2}
  \begin{pmatrix}
 0 & -1 \\ 1 & 0
  \end{pmatrix}, \quad
  \widehat{u}_3 = \frac{1}{2}
  \begin{pmatrix}
 -i & 0 \\ 0 & i
  \end{pmatrix},
\end{equation}
are a standard Milnor basis.
\end{exa}

One has the following theorem which was originally proven by J.~Milnor. 

\begin{theo}[{Similar to \cite[Lemma~4.1]{Milnor1976}, see also \cite[Section~2.3]{EldredgeGordinaSaloff-Coste2018}}] \label{milnor-basis-su2}\label{standard-milnor}
Given any inner product $g$ on $\mathfrak{su}(2)$, there exists a  standard Milnor basis $(u_1, u_2, u_3)$ which is orthogonal with respect to $g$.  Moreover, if we set $a_i = \sqrt{ \langle u_i, u_i \rangle_{g}}$, $i=1,2,3$, then the standard Milnor basis $\{u_1, u_2, u_3\}$ can be chosen so that $a_1 \leqslant a_2 \leqslant a_3$.  Any two such choices result in the same ordered triple $(a_1, a_2, a_3)$, which we call the \textbf{parameters} of $g$, and this triple determines $g$ uniquely up to an isometric isomorphism.  
\end{theo}

For computational purposes, it is very convenient to work in a standard Milnor basis, because of the simplicity of the bracket relations \eqref{e.BasisLie Brackets}, $[u_i, u_j]=u_k$.  In particular, one obtains useful identities involving products of one-parameter subgroups $e^{t  u_i}$; see Section \ref{identities-sec} below.

We seek an analogous result for Lie algebras of the form $\mathfrak{g} = \mathfrak{su}(2) \oplus \mathfrak{z}$, with $\mathfrak{z}$ abelian.  The fundamental difficulty, as mentioned in Remark \ref{non-product}, is that for a generic inner product $g$ on $\mathfrak{g}$, the Lie subalgebras $\mathfrak{su}(2)$ and $\mathfrak{z}$ need not be orthogonal. Recall that we use $\oplus$ to denote the direct sum in the category of Lie algebras, not inner product spaces, so we do not intend it to
imply orthogonality.  When this is intended, we will write $\oplus_\perp$.

\begin{lem}\label{skewed-basis}
Let $g$ be an inner product on $\mathfrak{g} = \mathfrak{su}(2) \oplus \mathfrak{z}$, and let $\mathfrak{z}^\perp$ be the $g$-orthogonal complement of $\mathfrak{z}$.  There exist orthogonal
  vectors $v_1, v_2, v_3 \in \mathfrak{z}^\perp$ which can be written as 
\begin{equation}\label{v-u-h}
  v_i = u_i + h_i,
\end{equation}
where $u_1, u_2, u_3 \in \mathfrak{su}(2)$ is  a standard Milnor basis, and $h_1, h_2, h_3 \in \mathfrak{z}$.
\end{lem}

\begin{proof}
Let $P: \mathfrak{g} \to \mathfrak{z}^\perp $ be the orthogonal projection onto $\mathfrak{z}^\perp $, and consider the antisymmetric bilinear operation $[\cdot, \cdot]'$ on $\mathfrak{z}^\perp$ defined by $[u,v]^{\prime} = P[u,v]$.

We claim that $(\mathfrak{z}^\perp , [\cdot, \cdot]^{\prime})$ is a Lie algebra isomorphic to $\mathfrak{su}(2)$.  First, for any $u, v \in \mathfrak{g}$, we have $P u - u, P v - v \in \mathfrak{z}$, where $\mathfrak{z}$ is the center of $\mathfrak{g}$, and  thus
\begin{equation*}
 [P u, Pv]^{\prime} = P[Pu, Pv] = P[u + (Pu - u), v + (Pv-v)] = P[u,v]
\end{equation*}
so that $P$ is a homomorphism of non-associative algebras. Since $P$ has kernel $\mathfrak{z}$ and image $\mathfrak{z}^\perp$, we have $(\mathfrak{z}^\perp, [\cdot,\cdot]^{\prime}) \cong \mathfrak{g}/\mathfrak{z} \cong \mathfrak{su}(2)$.

As such, Theorem~\ref{standard-milnor} implies that $(\mathfrak{z}^\perp , [\cdot, \cdot]^{\prime})$ admits a $g$-orthogonal standard Milnor basis $v_1, v_2, v_3 \in \mathfrak{z}^\perp $, so that $P[v_i, v_j] = v_k$ for $(i,j,k)$ being any cyclic permutation of $(1,2,3)$.  Since $\mathfrak{g} = \mathfrak{su}(2) \oplus \mathfrak{z}$  is a (non-orthogonal) direct sum, we can uniquely write $v_i = u_i + h_i$, where $u_i \in \mathfrak{su}(2)$ and $h_i \in \mathfrak{z}$.  It remains to show that $u_1, u_2, u_3$ form a standard Milnor basis with respect to $[\cdot, \cdot]$.  Note that $P u_i = P v_i = v_i$.  Since $P$ is an isomorphism of Lie algebras, we have
\begin{equation*}
 P[u_i, u_j] = [Pu_i, Pu_j]^{\prime} = [v_i, v_j]^{\prime} = v_k = P u_k
\end{equation*}
so that $[u_i, u_j] - u_k \in \mathfrak{z}$.  But $\mathfrak{su}(2) \cap \mathfrak{z} = 0$ so $[u_i, u_j]= u_k$.
\end{proof}

\subsection{A lifting construction}

The remaining difficulty is that the vectors $h_i$ in \eqref{v-u-h} need not be orthogonal to each other; they may even be linearly dependent.  However, by means of the following lifting procedure, we may reduce to the case where they are orthogonal, and even of equal length.

\begin{lem}\label{lift-euclidean}
Let $h_1, \dots, h_k \in \mathbb{R}^n$ be a collection of vectors, and let $\pi : \mathbb{R}^{n+k} \to \mathbb{R}^n$ be orthogonal projection onto the first $n$ coordinates.  Then there exist orthonormal vectors $f_1, \dots, f_k \in \mathbb{R}^{n+k}$ and a scalar $d$ such that $\pi(d f_i) = h_i$.
\end{lem}

\begin{proof}
First denote $d f_i = (h_i, w_i)$ for some $w_1, \dots, w_k \in \mathbb{R}^k$.  The condition that the $f_i$ should be orthonormal then reads $\langle h_i, h_j \rangle + \langle w_i, w_j \rangle = d^2 \delta_{ij}$, or equivalently $\langle w_i, w_j \rangle =d^2 \delta_{ij} - \langle h_i, h_j \rangle$.  Such vectors $w_i$ can  be found if the symmetric matrix $(d^2 \delta_{ij} - \langle h_i, h_j \rangle)_{i,j=1}^k$ is positive semi-definite. This is the case if $d$ is chosen sufficiently large, namely, at least as large as the square root of the largest eigenvalue of the positive semi-definite matrix $(\langle h_i,  h_j \rangle)_{i,j=1}^k$.
\end{proof}

This motivates the following definition for the type of basis we will consider.

\begin{defin}\label{decoupled-def}
Let $\mathfrak{g}$ be a Lie algebra of the form $\mathfrak{g} = \mathfrak{su}(2) \oplus \mathfrak{z}$, where $\mathfrak{z}$ is abelian, and let $\langle \cdot, \cdot \rangle_{g}$ be an inner product on $\mathfrak{g}$.  We will say an ordered basis $(v_1, v_2, v_3, f_1, \dots f_{3+n})$ of $\mathfrak{g}$ is a \emph{decoupled Milnor basis} for $g$ if it satisfies the following properties.
\begin{enumerate}
\item The vectors $\{v_1, v_2, v_3, f_1, \dots f_{3+n}\}$ are $g$-orthogonal;
\item The vectors $\{f_1, \dots, f_{3+n}\}$ form a $g$-orthonormal basis for $\mathfrak{z}$;
\item There exists a scalar $d \in \mathbb{R}$ such that the vectors $v_1, v_2, v_3$ can be written as $v_i = u_i + d f_i$, where $u_1, u_2, u_3 \in \mathfrak{su}(2)$ is a standard Milnor basis.
\end{enumerate}
If $g$ admits a decoupled Milnor basis, we will simply say that $g=\langle \cdot, \cdot \rangle_{g}$ is a \emph{decoupled inner product}. 

We denote by $\mathcal{L}_{\operatorname{dec}}(\mathfrak{g})$ the set of all decoupled inner products on $\mathfrak{g}$, and by $\mathcal{L}_{\operatorname{dec}}(G)$ the set of all decoupled left-invariant Riemannian metrics on a Lie group $G$ whose Lie algebra $\mathfrak{g}$ is of this form.
\end{defin}

Having the coefficient $1$ on $u_i$ in the expression $v_i = u_i + d f_i$ will be very convenient in later computations, which is why we have chosen this normalization in which the $v_i$ need not be unit vectors.

For a given decoupled inner product $g$, there may be more than one decoupled Milnor basis $(v_1, v_2, v_3, f_1, \dots f_{3+n})$ satisfying the properties of Definition \ref{decoupled-def}, but our computations below will turn out to be independent of the choice of such a basis. \label{basis-choice}

\begin{lem}\label{lift-to-decoupled}
Let $\mathfrak{g} = \mathfrak{su}(2) \oplus \mathfrak{z}$, where $\mathfrak{z}$ is an abelian Lie algebra of dimension $n$, and let $g$ be an inner product on $\mathfrak{g}$.  Then there exists $\mathfrak{g}^\prime, g^\prime, \pi$ such that:
  \begin{itemize}
  \item $\mathfrak{g}^{\prime}$ is a Lie algebra of the form $\mathfrak{g}^{\prime} = \mathfrak{su}(2) \oplus \mathfrak{z}^{\prime}$, with $\mathfrak{z}^{\prime}$ abelian and $\dim \mathfrak{z}^{\prime} = \dim \mathfrak{z} + 3$;
  \item $g^{\prime}$ is a decoupled inner product on $\mathfrak{g}^{\prime}$; and
  \item $\pi : \mathfrak{g}^{\prime} \to \mathfrak{g}$ is a surjective
    Lie algebra homomorphism which is a partial isometry with respect
    to the inner products $g^{\prime}, g$.
  \end{itemize}
\end{lem}

\begin{proof}
Let $v_i = u_i + h_i \in \mathfrak{g}$, $i=1,2,3$, be as in Lemma~\ref{skewed-basis}, and let $a_i = \sqrt{\langle v_i, v_i \rangle_{g}}$.  Identify $(\mathfrak{z}, \left.g\right|_{\mathfrak{z}})$ with $\mathbb{R}^n$, and let $\mathfrak{z}^{\prime} = \mathbb{R}^{n+3}$ equipped with the Euclidean inner product.  Let $\pi_{0} : \mathfrak{z}^{\prime} \to \mathfrak{z}$ be the projection onto the first $n$ coordinates, which is a partial isometry.  Let $\mathfrak{g}^{\prime} = \mathfrak{su}(2) \oplus \mathfrak{z}^{\prime}$ and $\pi(u+f) = u+\pi_0(f)$ for $u \in  \mathfrak{su}(2)$, $f \in \mathfrak{z}^{\prime}$.

By Lemma~\ref{lift-euclidean}, we can find $f_1, f_2, f_3$ orthonormal in $\mathfrak{z}^{\prime}$ and $d \in \mathbb{R}$ so that $\pi(d f_i) = h_i$.  Set $v_i^{\prime} = u_i + d f_i$, so that $\pi(v_i^{\prime}) = v_i$. Let $g^{\prime}$ be the unique inner product on $\mathfrak{g}^{\prime}$ which agrees with the Euclidean inner product on $\mathfrak{z}^{\prime}$ and for which the vectors $\frac{1}{a_i} v_i^{\prime}$ are orthonormal and orthogonal to $\mathfrak{z}^{\prime}$.  Then $\pi$ is a partial isometry of $(\mathfrak{g}^{\prime}, g^{\prime})$ to $(\mathfrak{g}, g)$: it maps the $g^\prime$-orthonormal basis $(\frac{1}{a_1} v_1^{\prime}, \frac{1}{a_2} v_2^{\prime}, \frac{1}{a_3} v_3^{\prime}, e_1, \dots, e_{n+3})$ to $(\frac{1}{a_1} v_1, \frac{1}{a_2} v_2, \frac{1}{a_3} v_3, e_1, \dots, e_{n}, 0, 0, 0)$.  Finally, to see that the inner product $g^\prime$ is decoupled, we can extend $f_1, f_2, f_3$ to an orthonormal basis $f_1, \dots, f_{n+3}$ for $\mathfrak{z}^{\prime}$; then $(v_1^{\prime}, v_2^{\prime}, v_3^{\prime}, f_1, \dots, f_{n+3})$ is a decoupled Milnor basis for $(\mathfrak{g}^\prime,g^\prime)$.
\end{proof}

The last step is to note that when we have a decoupled inner product on $\mathfrak{g} = \mathfrak{su}(2) \oplus \mathfrak{z}$, then all but $3$ dimensions of $\mathfrak{z}$ are essentially irrelevant to the questions at hand,  and therefore they can be dispensed with.

\begin{lem}\label{dec-product}
Let $g$ be a decoupled inner product on $\mathfrak{g} = \mathfrak{su}(2) \oplus \mathfrak{z}$, where $\mathfrak{z}$ is abelian and $\dim \mathfrak{z} \geqslant 3$.  Then $(\mathfrak{g}, g)$ decomposes as an orthogonal Lie algebra direct sum $\mathfrak{g} = \mathfrak{g}_1 \oplus_\perp \mathfrak{g}_2$ such that
\begin{itemize}
  \item $\mathfrak{g}_1$ is isomorphic to $\mathfrak{su}(2) \oplus
 \mathbb{R}^3$;
  \item The restriction $g_1$ of the inner product $g$ to
 $\mathfrak{g}_1$ is decoupled; and
  \item $\mathfrak{g}_2$ is abelian.
  \end{itemize}
\end{lem}

\begin{proof}
Set $n := \dim \mathfrak{z} \geqslant 3$, and let $(v_1, v_2, v_3, f_1, \dots, f_n)$ be a decoupled Milnor basis for $(\mathfrak{g}, g)$, which in particular is an orthogonal basis.  Set
\begin{align*}
\mathfrak{g}_1 &= \operatorname{span} \{v_1, v_2, v_3, f_1, f_2,
 f_3\} 
\\
\mathfrak{g}_2 &= \operatorname{span} \{f_4, \dots, f_n\}.
\end{align*}
By construction, $\mathfrak{g}$ is the orthogonal direct sum of the  linear subspaces $\mathfrak{g}_1, \mathfrak{g}_2$.

Now since $u_i = v_i - d_i f_i \in \mathfrak{g}_1$ for  $i=1,2,3$, we have $\mathfrak{su}(2) = [\mathfrak{g}, \mathfrak{g}]  \subset \mathfrak{g}_1$, and so $\mathfrak{g}_1$ is a Lie subalgebra of $\mathfrak{g}$.  Clearly $\mathfrak{g}_1 \cong \mathfrak{su}(2) \oplus \mathbb{R}^3$.  Moreover, the restriction of $g$ to $\mathfrak{g}_1$ is again decoupled, since $(v_1, v_2, v_3, f_1, f_2, f_3)$ is a decoupled Milnor  basis for $g|_{\mathfrak{g}_1}$.

Finally, since $\mathfrak{g}_2 \subset \mathfrak{z}$, it is trivial that $\mathfrak{g}_2$ is an abelian Lie subalgebra of   $\mathfrak{g}$ that commutes with $\mathfrak{g}_1$.  This completes the proof.
\end{proof}

Assembling together the previous lemmas, we can reduce our main theorem to showing that the decoupled metrics on $\operatorname{SU}(2) \times \mathbb{R}^3$ are uniformly doubling.

\begin{theo}\label{doubling-reduction}
Suppose the family of decoupled left-invariant Riemannian metrics $\mathcal{L}_{\operatorname{dec}}\left( \operatorname{SU}(2) \times \mathbb{R}^3 \right)$ is uniformly doubling with some constant $D_{\operatorname{dec}}\left( \operatorname{SU}(2) \times \mathbb{R}^{3} \right)$. Then each Lie group $G$ of the form $G =\operatorname{SU}(2) \times \mathbb{R}^n$ is uniformly doubling, with constant at most $(2^n D_{\operatorname{dec}}(\operatorname{SU}(2) \times \mathbb{R}^3))^{4}$.  Moreover, every quotient $H = G / S$ of such a group by a (closed) normal subgroup is uniformly doubling, with the same bound on its constant.
\end{theo}

In other words, if $\mathcal{L}_{\operatorname{dec}}(\operatorname{SU}(2) \times \mathbb{R}^3)$ is uniformly doubling, then every connected Lie group $H$ whose universal cover is of the form $\operatorname{SU}(2) \times \mathbb{R}^n$ is likewise uniformly doubling, with a bound depending only on $n$.  This includes the unitary group $\operatorname{U}(2)$, the special orthogonal group
$\operatorname{SO}(3)$, and every group $\operatorname{SU}(2) \times Z$ with $Z$ being connected and abelian.

\begin{proof}
Fix $n$ and let $g$ be an arbitrary left-invariant Riemannian metric on $G = \operatorname{SU}(2) \times \mathbb{R}^n$; as before, we also use $g$ for the corresponding inner product on the Lie algebra $\mathfrak{g} = \operatorname{su}(2) \oplus \mathbb{R}^n$. By Lemma~\ref{lift-to-decoupled} there is a decoupled inner product $g^{\prime}$ on $\mathfrak{g}^{\prime}:= \mathfrak{su}(2) \oplus \mathbb{R}^{3+n}$ and a surjective Lie algebra homomorphism $\pi : \mathfrak{g}^{\prime} \to \mathfrak{g}$ which is a partial isometry of $g^{\prime}$ to $g$.  Let $G^{\prime} =  \operatorname{SU}(2) \times \mathbb{R}^{3+n}$, on which $g^{\prime}$ may be viewed as a left-invariant Riemannian metric.  As both $G$ and $G^{\prime}$ are connected and simply connected, $\pi$ lifts to a smooth homomorphism $\Pi: G^{\prime} \to G$ with $d\Pi_e = \pi$.  Since $\pi$ was a partial isometry, $\Pi$ is a Riemannian submersion of the metrics $g^{\prime}, g$.

Now as shown by Lemma~\ref{dec-product}, the Lie algebra $\mathfrak{g}^{\prime}$ decomposes orthogonally under $g^{\prime}$ as $\mathfrak{g}^{\prime} = \mathfrak{g}_1^{\prime} \oplus_\perp \mathfrak{g}_2^{\prime}$, where the inner product on $\mathfrak{g}_1^{\prime} \cong \mathfrak{su}(2) \oplus \mathbb{R}^3$ is decoupled, and $\mathfrak{g}_2^{\prime} \cong \mathbb{R}^n$ is abelian.
Hence $g^{\prime}$ is a product metric on $G^{\prime} = G_1^{\prime} \times G_2^{\prime}$, whose  restriction to $G^{\prime} \cong \operatorname{SU}(2) \times \mathbb{R}^3$ is decoupled, and where $G_2^{\prime} \cong \mathbb{R}^n$ is abelian.  Since every left-invariant Riemannian metric on $G_2$ is isometric to the Euclidean metric on $\mathbb{R}^n$, and has  doubling constant $2^n$, by Proposition~\ref{doubling-product} we have
\begin{equation*}
D_{g^{\prime}} \leqslant (2^n D_{g_1^{\prime}})^2 \leqslant (2^n D_{\operatorname{dec}}(\operatorname{SU}(2) \times \mathbb{R}^3))^2.
\end{equation*}
By an application of Theorem~\ref{doubling-quotient}, we now have
  \begin{equation*}
 D_g \leqslant D_{g^{\prime}}^2 \leqslant (2^n D_{\operatorname{dec}}(\operatorname{SU}(2) \times \mathbb{R}^3))^4.
  \end{equation*}
 Since $g$ was an arbitrary left-invariant metric on $G$, we conclude that $G$ is uniformly doubling.

Now suppose $H = G / S$ is a quotient of $G$, where $S \leqslant G$ is a closed normal subgroup, and let $\varphi : G \to H$ be the quotient map.  If $h$ is any left-invariant metric on $H$, then by Lemma~\ref{quotient-lift} we have a left-invariant metric $g$ on $G$ such that $\varphi$ is a Riemannian submersion from $(G, g)$ to $(H,h)$.  Constructing $(G^{\prime}, g^{\prime})$ and $\pi : G^{\prime} \to G$ as above,  we have that $\varphi \circ \pi : (G^{\prime}, g^{\prime}) \to (H,h)$ is a group homomorphism and a Riemannian submersion, so we can again  apply Theorem~\ref{doubling-quotient} to get the conclusion
\begin{equation*}
 D_h \leqslant D_{g^{\prime}}^2 \leqslant (2^n D_{\operatorname{dec}}(\operatorname{SU}(2) \times \mathbb{R}^3))^4.
\end{equation*}
\end{proof}

\subsection{The bi-invariant reference metric}

It will be convenient to make use of the following fixed reference metric and corresponding measure on $\operatorname{SU}(2) \times \mathbb{R}^3$.

\begin{notation}\label{g0-def}
Let $g_0$ be the $\operatorname{Ad}$-invariant inner product on $\mathfrak{g} = \mathfrak{su}(2) \oplus \mathbb{R}^3$ defined by
\begin{equation*}
\langle u+f, u^{\prime}+f^{\prime}\rangle_{g_{0}} = -2 \tr(u u^{\prime}) + \langle f, f^{\prime}\rangle_{\mathbb{R}^{3}},
\end{equation*}
where $u, u^{\prime} \in \mathfrak{su}(2)$ are realized as $2 \times 2$ complex matrices, $f, f^{\prime} \in \mathbb{R}^3$, and $\langle \cdot,   \cdot \rangle_{\mathbb{R}^{3}}$ is the usual Euclidean inner product on $\mathbb{R}^3$.  We also denote by $g_0$ the corresponding  bi-invariant Riemannian metric on $G = \operatorname{SU}(2) \times  \mathbb{R}^3$.  Finally, let  $\mu_0$ be the Riemannian volume measure of $g_0$, which is simply a particular normalization of the  bi-invariant Haar measure on $G$.
\end{notation}

\begin{pro}\label{milnor-orthonormal}
Any standard Milnor basis $u_1, u_2, u_3 \in \mathfrak{su}(2)$ is orthonormal with respect to $g_0$.
\end{pro}

\begin{proof}
This follows from \cite[Lemma 2.7]{EldredgeGordinaSaloff-Coste2018}, where it is shown that, as matrices, $u_i^2 = -\frac{1}{4} I$ and $u_i u_j = \frac{1}{2} u_k$.  Thus $\langle u_i, u_i \rangle_{g_{0}} = -\frac{1}{2}\tr(I) = 1$, and $\langle u_i, u_j\rangle_{g_{0}} = - \tr(u_k) = 0$ since every $u \in \mathfrak{su}(2)$ has zero trace.
\end{proof}

\begin{rem}\label{g0-decoupled}
As a consequence, $g_0$ is decoupled: if $u_1, u_2, u_3$ is a standard Milnor basis for $\mathfrak{su}(2)$ and $f_1, f_2, f_3$ is an orthonormal basis for $\mathbb{R}^3$, then $(u_1, u_2, u_3, f_1, f_2, f_3)$ satisfies the definition of a decoupled Milnor basis, with $v_i = u_i$ and $d = 0$.
\end{rem}

\subsection{Parametrization of decoupled metrics}

There is a convenient parametrization of the decoupled inner products on $\mathfrak{su}(2) \oplus \mathbb{R}^3$, which can be seen as an extension of the parametrization for inner products on
$\mathfrak{su}(2)$ used in \cite[Section 2.3]{EldredgeGordinaSaloff-Coste2018}.

\begin{defin}\label{parameters}
  The \emph{parameters} of a decoupled inner product $g$ on $\mathfrak{g} = \mathfrak{su}(2) \oplus \mathbb{R}^3$, with respect to a corresponding decoupled Milnor basis $( v_1, v_2, v_3, f_1, f_2, f_3)$, are the 4-tuple of numbers $(a_1, a_2, a_3, d)$, where $a_i = \sqrt{\langle v_i, v_i \rangle_{g}}$, and $d \in \mathbb{R}$ is as in Definition
\ref{decoupled-def} so that $u_i := v_i - d f_i$ form a standard Milnor basis for $\mathrm{su}(2)$.
\end{defin}

\begin{pro}\label{parameters-well-defined}
  The parameters $(a_1, a_2, a_3, d)$ of a given decoupled inner product $g$ are independent of the choice of decoupled Milnor basis, up to a reordering of the $a_i$ and a sign change of $d$.  Moreover, the basis can be
chosen so that $a_1 \leqslant a_2 \leqslant a_3$ and $d \geqslant 0$.
\end{pro}

\begin{proof}
Let $(v_i, f_i), (v_i^\prime, f_i^\prime)$ be two decoupled Milnor bases for $g$, with corresponding parameters $(a_i, d), (a_i^\prime, d^\prime)$.   We show that $a_i = a_i^{\prime}$ up to reordering, and $|d| = |d^{\prime}|$.

Let $u_i := v_i - d f_i$, so that $(u_1, u_2, u_3)$ is a standard Milnor basis for $\mathfrak{su}(2)$, and define $u_i^\prime$ likewise.  By Proposition \ref{milnor-orthonormal}, the bases $(u_i)$ and $(u_i^\prime)$ are each  orthonormal for the $\operatorname{Ad}$-invariant metric $g_0$ (see Notation \ref{g0-def}),  so we can write $u_1^\prime = \sum_{i=1}^3 b_i u_i$ where $\sum_{i=1}^3 b_i^2 = 1$.  Likewise, since the $(f_i)$ are orthonormal for $g$, as are the $(f_i^\prime)$, we can write $f_1^\prime = \sum_{j=1}^3 c_j f_j$ with $\sum_{j=1}^3 c_j^2 = 1$.  Now, we have
\begin{equation*}
\langle u_i, f_j \rangle_{g} = \langle v_i - d f_i, f_j \rangle_{g} = -d \cdot \delta_{ij}
\end{equation*}
and so
\begin{align*}
  |d^\prime| &= \left|\langle u_1^\prime, f_1^\prime\rangle_{g}\right| = \left| \sum_{i,j=1}^3 b_i c_j \langle u_i, f_j\rangle_{g} \right| = \left|- d \sum_{i=1}^3 b_i c_i \right| \leqslant |d|
\end{align*}
by the Cauchy--Schwarz inequality.  The opposite inequality also holds by symmetry, so $|d^{\prime}| = |d|$; that is, $d^{\prime} = \pm d$.

Next, define the linear transformation $A$ of $\mathfrak{su}(2)$ by $\langle Au, w \rangle_{g_{0}} = \langle u, w\rangle_{g}$ for $u, w \in \mathfrak{su}(2)$.
Now we have
\begin{align*}
\langle A u_i, u_j \rangle_{g_{0}} &= \langle u_i, u_j\rangle_{g} = \langle v_i - d f_i, v_j - d f_j \rangle_{g} = (a_i^2 + d^2)\delta_{ij}
\end{align*}
so that $A u_i = (a_i^2 + d^2) u_i$.  That is, the $u_i$ are eigenvectors of $A$ with corresponding eigenvalues $a_i^2 + d^2$.  The same argument applies to the primed basis, and so, up to reordering, we must have $a_i^2 + d^2 = (a_i^\prime)^2 + d^2$  and thus $a_i = a_i^\prime$.

To get the final assertion, it is easily checked that if $(v_1, v_2, v_3, f_1, f_2, f_3)$ is a decoupled Milnor basis for $g$, with parameters $(a_1, a_2, a_3, d)$, then so are the following ordered bases:
\begin{itemize}
\item $(v_2, v_3, v_1, f_2, f_3, f_1)$, with parameters $(a_2, a_3, a_1, d)$;
\item $(-v_2, -v_1, -v_3, -f_2, -f_1, -f_3)$, with parameters $(a_2, a_1, a_3, d)$;
\item $(v_1, v_2, v_3, -f_1, -f_2, -f_3)$, with parameters $(a_1, a_2, a_3, -d)$.
\end{itemize}
Iterating these transformations as needed, we see that we can attain any ordering of the $a_i$, and either sign of $d$.
\end{proof}

The following proposition extends \cite[Corollary 2.10]{EldredgeGordinaSaloff-Coste2018}.

\begin{pro}\label{up-to-isometry}
On $G = \operatorname{SU}(2) \times \mathbb{R}^{3}$, for every set of  parameters $0 < a_1 \leqslant a_2 \leqslant a_3 < \infty$ and $d \geqslant 0$, there exists a decoupled left-invariant Riemannian metric $g$ with these parameters, and such $g$ is unique up to an isometric Lie group isomorphism.
\end{pro}

\begin{proof}
The existence is trivial.  Suppose numbers $0 < a_1 \leqslant a_2 \leqslant a_3 < \infty$ and $d \geqslant 0$ are given.  Let $(u_1, u_2, u_3)$ be any standard  Milnor basis of $\mathfrak{su}(2)$, and $(f_1, f_2, f_3)$ any orthonormal basis of $\mathbb{R}^3$.  Set $v_i = u_i + d f_i$.  Then $\{v_1, v_2, v_3, f_1, f_2, f_3\}$ are linearly independent, so there is an inner product $g$ for which the vectors $\{\frac{1}{a_i} v_i, f_i\}$ are orthonormal.  By construction, this $g$ has parameters $(a_1, a_2, a_3, d)$.

To show uniqueness, let $g, g^{\prime}$ be two decoupled left-invariant Riemannian metrics on $G = \operatorname{SU}(2) \times \mathbb{R}^3$, identified with inner products  on $\mathfrak{g} = \mathfrak{su}(2) \oplus \mathbb{R}^3$, and both  having parameters $a_1, a_2, a_3, d$.  Let $(v_1, v_2, v_3, f_1, f_2, f_3)$ be a decoupled Milnor basis for $g$, so that $u_i = v_i - d f$ are a standard Milnor basis for $\mathfrak{u}(2)$.  Define $v_i^{\prime}, f_i^{\prime}, u_i^{\prime}$ analogously for $g^{\prime}$.

Define a linear automorphism $\phi : \mathfrak{g} \to \mathfrak{g}$ by $\phi(v_i) = v_i^{\prime}$, $\phi(f_i) = f_i^{\prime}$, so that also $\phi(u_i) = u_i^{\prime}$.  Then $\phi$ is an isometry from $g$ to $g^{\prime}$, since it maps the $g$-orthonormal basis $( \frac{1}{a_i} v_i, f_i)$ to the $g^\prime$-orthonormal basis $(\frac{1}{a_i} v_i^\prime, f_i^\prime)$.
Moreover, $\phi$ is a Lie algebra automorphism of $\mathfrak{g}$.  To see this, note that $(u_i, f_i)$ is a basis for $\mathfrak{g}$.  Since the $u_i$ and $u_i^{\prime}$ are both standard Milnor bases, for any cyclic permutation $(i,j,k)$ of $(1,2,3)$, we have
\begin{equation*}
[\phi(u_i), \phi(u_j)] = [u_i^{\prime}, u_j^{\prime}] = u_k^{\prime} = \phi(u_k) =\phi([u_i, u_j]).
\end{equation*}
And for all $i,j = 1,2,3$ we have
\begin{align*}
[\phi(u_i), \phi(f_j)] &= [u_i^{\prime}, f_j^{\prime}] = 0 = \phi([u_i, f_j]) 
\\
[\phi(f_i), \phi(f_j)] &= [f_i^{\prime}, f_j^{\prime}] = 0 = \phi([f_i, f_j]).
\end{align*}
Now since $G = \operatorname{SU}(2) \times \mathbb{R}^3$ is connected and simply connected, the Lie algebra automorphism $\phi : \mathfrak{g} \to \mathfrak{g}$ lifts to a Lie group automorphism $\Phi: G \to G$ with $d\Phi_e = \phi$.  As $\phi$ is an isometry and  the metrics $g, g^{\prime}$ are both left-invariant, we have that $\Phi$ is a Riemannian isometry.
\end{proof}

\begin{rem}\label{no-basis-dependence}
  As a consequence of Proposition~\ref{up-to-isometry}, we could, whenever convenient, assume without loss of generality that the  vectors $u_i$ in a decoupled Milnor basis are the Pauli matrices $\widehat{u}_i$, and the vectors $f_i$ are the standard Euclidean  basis for $\mathbb{R}^3$.  This also shows that the dependence of geometric  quantities such as doubling constant, volume growth function, etc., on an arbitrary decoupled metric $g$  can be expressed in terms of the parameters $a_1, a_2, a_3, d$.
\end{rem}

\section{Statement of volume estimates}\label{volume-statement}

With the notation of the previous section, we can now state the volume estimate which is the main technical theorem of this paper.  To state it more concisely, we introduce some auxiliary notation.

\begin{notation}\label{m-rho-def}
  Given a decoupled left-invariant metric $g$ having parameters $(a_1, a_2, a_3, d)$ and a constant $\eta > 0$, let
  \begin{align}
    m_i &= m_i(r, a_i, \eta) = \min\left(\frac{r}{a_i}, \eta\right) \label{mi-def} \\ 
    \rho_i &=  \rho_i(r, a_1, a_2, a_3, \eta) = m_j m_k + m_i m_j^2 + m_i m_k^2 \label{rhoi-def}
  \end{align}
  where $(i,j,k)$ denotes any cyclic permutation of the indices $(1,2,3)$.
\end{notation}

The quantity $m_i$ estimates the maximum $u_i$ component of the path $e^{t v_i}$, $0 \leqslant t \leqslant r/a_i$, whose $g$-length is $r$.  Because of the compactness of $\mathrm{SU}(2)$, it cuts off when $r/a_i$ exceeds some value $\eta$.  The terms of $\rho_i$ correspond to the bracket relations
\begin{equation}
  u_i = [v_j, v_k] = [v_j, [v_i, v_j]] = [[v_k, v_i], v_k].
\end{equation}

Let $\mu_0$ denote the normalized Haar measure on $\operatorname{SU}(2) \times \mathbb{R}$ as defined in Notation~\ref{g0-def}. 

\begin{theo}\label{volume-estimate}
There are universal constants $0 < c < C < \infty$ and $\eta > 0$ such that for every decoupled left-invariant Riemannian metric $g \in \mathcal{L}_{\operatorname{dec}}(\operatorname{SU}(2) \times \mathbb{R}^{3})$, we have
\begin{equation*}
 c \overline{V}_g(r) \leqslant \mu_0(B_g(r)) \leqslant C \overline{V}_g(r),
\end{equation*}
where
\begin{align*}
  \overline{V}_g(r) &:= \prod_{i=1}^3 \min\left(d \rho_i \frac{r}{a_i} + \rho_i r + \frac{r^2}{a_i}, d \frac{r}{a_i} + r\right)
\end{align*}
where $(a_1, a_2, a_3, d)$ are the parameters of the metric $g$ as defined in Definition \ref{parameters}, and $\rho_i = \rho_i(r, a_1, a_2, a_3, \eta)$ is as defined as in Notation \ref{m-rho-def}.
\end{theo}

\begin{rem}
  The value of the universal constant $\eta$ is actually immaterial in the statement of Theorem \ref{volume-estimate}, as it may be absorbed into the constants $c,C$.  We include it in the statement to match the structure of the proof, but for purposes of applying Theorem \ref{volume-estimate}, one could replace $\eta$ with $1$ or any other fixed positive number.  
\end{rem}

The proof of Theorem \ref{volume-estimate} occupies the remainder of the paper, as we outlined in the introduction.

To see that Theorem~\ref{volume-estimate} implies uniform doubling, it suffices to verify that we have $\overline{V}_g(2r) \leqslant \overline{D} \cdot \overline{V}_g(r)$ for some constant $\overline{D}$ indepedent of $r, a_1, a_2, a_3, d$.   This can be done explicitly, but it is tedious, so instead we use a more concise abstract argument to show that such $\overline{D}$ exists.

For the following, let $\preceq$ denote the domination partial ordering on $(\mathbb{R}^+)^k$, i.e. $(x_1, \dots, x_k) \preceq (y_1, \dots, y_k)$ if and only if $x_i \leqslant y_i$ for all $i=1,\dots,k$.  When $k=1$ this is the usual total ordering on $\mathbb{R}^+$.  Also, let
$\Theta$ be an arbitrary set, thought of as a space of parameters $\theta$.  For the purposes of Theorem \ref{volume-estimate}, we can take $\Theta = \mathbb{R}^4$, with $\theta = (a_1, a_2, a_3, d)$ representing the parameters of a metric
$g$.

\begin{defin}
Suppose $f : (\mathbb{R}^+)^k \times \Theta \to (\mathbb{R}^+)^m$ is monotone increasing in its first argument, i.e. whenever $x \preceq  y$ we have $f(x; \theta) \preceq f(y; \theta)$ for each $\theta \in \Theta$.  We say $f$ is \textbf{uniformly doubling} if there is a  constant $D_f$ such that $f(2x; \theta) \preceq D_f f(x; \theta)$ for all $x \in (\mathbb{R}^+)^k$ and all $\theta \in \Theta$.
\end{defin}

\begin{lem}\label{unif-doubling-abstract}
The following functions are uniformly doubling. 
\begin{enumerate}
\item Any function $f(x ; \theta) = g(\theta)$ that depends only on $\theta$, with $D_f = 1$.
\item The sum function $f(x_1, \dots, x_k) = x_1 + \dots + x_k$  with $D_f = 2$.  Here we take $m=1$ and the function $f$ does not depend on $\theta$.
\item The product function $f(x_1, \dots, x_k) = x_1 \cdots x_k$ with $D_f = 2^k$.
\item The minimum function $f(x_1, \dots, x_k) = \min(x_1, \dots, x_k)$, with $D_f = 2$, and the maximum function likewise.
\item For $i=1,\dots, n$, suppose we have functions $f_i : (\mathbb{R}^+)^k \times \Theta \to (\mathbb{R}^+)^{m_i}$, each of which is uniformly doubling with some constant $D_i$.  Let $m = m_1 + \dots + m_n$.  Then the
   couple function $f : (\mathbb{R}^+)^k \times \Theta \to (\mathbb{R}^+)^m$ defined by $f(x; \theta) = (f_1(x ; \theta), \dots, f_m(x ; \theta))$ is uniformly doubling with constant $D_f = \max(D_1, \dots, D_m)$.
\item Suppose $g : (\mathbb{R}^+)^k \times \Theta \to (\mathbb{R}^+)^\ell$ and $h : (\mathbb{R}^+)^\ell \times \Theta \to (\mathbb{R}^+)^m$ are each uniformly doubling.  Then the composition $f(x; \theta) = h(g(x;\theta);\theta)$ is uniformly doubling  with constant $D_f \leqslant D_h^{\lceil \log_2 D_g \rceil}$.
\item If $f,g : (\mathbb{R}^+)^k \times \Theta \to (\mathbb{R}^+)^m$ are monotone increasing, $g$ is uniformly doubling with constant $D_g$, and there are constants $c, C$ such that $cg \leqslant f \leqslant Cg$, then $f$ is uniformly doubling with constant $D_f \leqslant \frac{C}{c} D_g$.
  \end{enumerate}
\end{lem}

The proofs are straightforward.

Since the function $\overline{V}_g(r) = \overline{V}(r ; a_1, a_2, a_3, d)$ is written as a composition of addition, multiplication, and minimum for $r>0$ and the non-negative
parameters $a_1, a_2, a_3, d$, it follows from Lemma \ref{unif-doubling-abstract} that $\overline{V}_g(r)$ is a uniformly doubling function, and hence so is $\mu_0(B_g(r))$.  We thus conclude the following.

\begin{cor}\label{decoupled-uniformly-doubling}
The class $\mathcal{L}_{\operatorname{dec}}(\operatorname{SU}(2) \times\mathbb{R}^{3})$ of decoupled left-invariant Riemannian metrics on $\operatorname{SU}(2) \times \mathbb{R}^3$ is uniformly doubling.
\end{cor}

Together with the reduction from Theorem~\ref{doubling-reduction}, this implies Theorem \ref{main-intro}.

\section{Milnor basis identities in $\operatorname{SU}(2)$}\label{identities-sec}

In order to approximate the metric balls in $\operatorname{SU}(2) \times
\mathbb{R}^3$, we will make extensive use of the algebraic structure of $\operatorname{SU}(2)$ and its Lie algebra $\mathfrak{su}(2)$.  In this section, we derive several identities and formulas that will be
useful.

At the heart of these computations is the fact that a standard Milnor basis $(u_1, u_2, u_3)$ for $\mathfrak{su}(2)$ satisfies the matrix identities
\begin{equation}
  u_i^2 = -\frac{1}{4} I, \qquad u_i u_j = \frac{1}{2} u_k, \qquad u_i
  u_j + u_j u_i = 0
\end{equation}
for any cyclic permutation $(i,j,k)$ of the indices $(1,2,3)$ by \cite[Lemma 2.7]{EldredgeGordinaSaloff-Coste2018}.  In fact, most of these computations depend on the structure of $\mathfrak{su}(2)$ \emph{only} through these relations, and so we change our notation to make this clear.

\begin{notation}\label{XYZ-notation}
In this section, $X,Y,Z$ denote real or complex square matrices, of some arbitrary dimension, which satisfy the following relations.

\begin{equation}\label{XYZ-relations}
\begin{gathered}
XY = \frac{1}{2} Z, \qquad YZ = \frac{1}{2}X, \qquad ZX = \frac{1}{2} Y,
\\
XY+YX = YZ+ZY = ZX+XZ = 0, 
\\
X^2 = Y^2 = Z^2 = -\frac{1}{4} I.
\end{gathered}
\end{equation}
Note this implies that $[X,Y]=Z$, $[Y,Z]=X$, $[Z,X]=Y$.
\end{notation}

\subsection{Closed form expressions}
We start with proving several identities that might serve as a replacement for using the Baker-Campbell-Dynkin-Hausdorff formula. The next statement is similar to \cite[Lemma 2.17]{EldredgeGordinaSaloff-Coste2018}.

\begin{pro}\label{p.adjoint} For matrices $X,Y$ as in Notation  ~\ref{XYZ-notation}, we have
\[
e^{-sY} X e^{sY}=\cos s X + \sin s [X, Y], s \in \mathbb{R}.
\]
\end{pro}

\begin{proof}
Denote $f\left( s \right):=e^{-sY} X e^{sY}$ and $g\left( s \right):=\cos s X + \sin s [X, Y]$. Then

\begin{align*}
& f\left( s \right)=e^{-sY} X e^{sY},  f\left( 0 \right)=X,
\\
& f^{\prime}\left( s \right)=e^{-sY} [X, Y] e^{sY},  f^{\prime}\left( 0 \right)=[X, Y],
\\
& f^{\prime \prime}\left( s \right)=e^{-sY} [[X, Y], Y] e^{sY},
\end{align*}
and
\begin{align*}
& g\left( s \right)=\cos s X + \sin s [X, Y],  g\left( 0 \right)=X,
\\
& g^{\prime}\left( s \right)=-\sin s X + \cos s [X, Y],  g^{\prime}\left( 0 \right)=[X, Y],
\\
& g^{\prime \prime}\left( s \right)=-\cos s X - \sin s [X, Y].
\end{align*}
Note that

\begin{align*}
& f^{\prime \prime}\left( s \right)=e^{-sY} [[X, Y], Y] e^{sY}=-e^{-sY} X e^{sY} =-f\left( s \right).
\end{align*}
Also

\begin{align*}
& g^{\prime \prime}\left( s \right)=-g\left( s \right).
\end{align*}
Thus $f$ and $g$ are solutions to the same initial value problem for a second order differential equation, and therefore they are equal.
\end{proof}

\begin{rem}
Proposition~\ref{p.adjoint} is valid for any two elements $X,Y$ of an arbitrary Lie algebra which satisfy $[[X,Y], Y]=-X$.
\end{rem}

\begin{cor}\label{cor.5.4}
\begin{align*}
& e^{-sY} e^{tX} e^{sY}=\sum_{n=0}^{\infty} \frac{t^{n}}{n!}\left(e^{-sY} X e^{sY} \right)^{n}=\exp\left( t \left( \cos s X + \sin s [X, Y] \right) \right), s, t \in \mathbb{R}.
\end{align*}
\end{cor}

We need expanded versions of several identities for standard Milnor bases based on \cite{EldredgeGordinaSaloff-Coste2018}.  In particular, \cite[Lemma 2.7]{EldredgeGordinaSaloff-Coste2018} states that for any standard Milnor basis $( u_{1}, u_{2}, u_{3})$ defined by Definition~\ref{d.MilnorBasis} we have

\begin{align*}
& u_{i}^{2}=-\frac{1}{4}I,
\\
& u_{i}u_{j}=\frac{1}{2}u_{k},
\\
& u_{i}u_{j}+u_{j}u_{i}=0
\end{align*}
where $\left(i, j, k \right)$ is  any cyclic permutation of the indices  $\left(1, 2, 3 \right)$ and $i\not=j$. In addition we have the following Rodrigues' formula \cite[Lemma 2.13, Remark 2.14]{EldredgeGordinaSaloff-Coste2018}

\begin{align}\label{e.Rodrigues1}
& \exp \left( A\right)=\cos \rho I+\frac{\sin \rho}{\rho}A,
\notag
\\
& \rho^{2} = \operatorname{det} A=\frac{x^{2}+y^{2}+z^{2}}{4},
\\
& A=xu_{1}+yu_{2}+zu_{3} \in \mathfrak{su}\left( 2 \right).
\notag
\end{align}
Note that these identities use not only Lie bracket relations, but also the specific structure of the matrix Lie algebra we are considering.

\begin{defin}\label{l.2ndKindCoordinates}
We define the coordinates of the second kind  $\Psi_{\operatorname{SU}\left( 2 \right)}: \mathbb{R}^{3} \longrightarrow \operatorname{SU}\left( 2 \right)$ by
\[
\left( x_{1}, x_{2}, x_{3} \right) \longmapsto \exp\left(
x_{3}u_{3}\right) \exp\left( x_{2}u_{2}\right) \exp\left(
x_{1}u_{1}\right),
\]
where $( u_{1}, u_{2}, u_{3} )$ is a standard Milnor
basis of $\mathfrak{su}\left( 2 \right)$. 
\end{defin}

We skip the proof of the next elementary fact.
\begin{pro}[Injectivity in the coordinates of the second kind]\label{p.5.4} Suppose

\[
\Psi_{\operatorname{SU}\left( 2 \right)}\left( x_{1}, x_{2}, x_{3}\right)=\Psi_{\operatorname{SU}\left( 2 \right)}\left( y_{1}, y_{2}, y_{3}\right)
\]
for some $x_{1}, x_{2}, x_{3}, y_{1}, y_{2}, y_{3} \in \mathbb{R}$. Then either

\[
y_{1}=x_{1}+2\pi m, \quad y_{2}=x_{2}+2\pi k, \quad y_{3}=x_{3}+2\pi n
\]
for some $m, k, n \in \mathbb{Z}$, or

\[
y_{1}=x_{1}=\frac{\pi}{2}+2\pi m, \quad y_{2}=\frac{\pi}{2}+x_{2}+2\pi k, \quad y_{3}=\frac{\pi}{2}+x_{3}+2\pi n
\]
for some $m, k, n \in \mathbb{Z}$.
\end{pro}

\subsection{Maurer--Cartan forms of smooth curves}

\begin{defin}
Let $G$ be a Lie group, with its Lie algebra $\mathfrak{g}$ identified with the tangent space $T_e G$ at the identity.  For   $\gamma : [0,1] \to G$ a smooth curve in $G$, the \textbf{Maurer--Cartan form} of $\gamma$ is the smooth curve   $c_{\gamma} : [0,1]\to \mathfrak{g}$ defined by
\begin{equation}
 c_\gamma(t) = \left.\frac{d}{d\epsilon}\right|_{\epsilon = 0}
 L_{\gamma(t)^{-1}} \gamma(t+\epsilon) = dL_{\gamma(t)}^{-1} \gamma^{\prime}(t)
 \end{equation}
where $\mathfrak{g} = T_e G$ is identified with the tangent space to $G$ at the identity.
\end{defin}

Note that, for any left-invariant Riemannian metric $g$ on $G$, the length of $\gamma$ with respect to $g$ is given by $\ell_g[\gamma] = \int_0^1 \sqrt{\langle c_{\gamma}(t), c_\gamma(t)\rangle_{g}}\,dt$.

\begin{pro}\label{p.4.7} Suppose $\gamma: [0, 1] \longrightarrow \operatorname{SU}\left( 2 \right)$ is a smooth curve such that
\[
\gamma\left( t \right)=\sigma\left( t \right) e^{h\left( t \right)u_{i}},
\]
where $\sigma: [0, 1] \longrightarrow \operatorname{SU}\left( 2 \right)$ is a smooth curve. Then the Maurer-Cartan form for the path $\gamma$ is given by

\[
c_{\gamma}\left( t \right)=
\left(
  \begin{array}{ccc}
   1 & 0 & 0 \\
   0 &\cos h & \sin h  \\
 0& -\sin h & \cos h  \\
  \end{array}
\right)
\left(
  \begin{array}{c}
 c_{i}  \\
 c_{j}  \\
 c_{k} \\
  \end{array}
\right)
+\left(
  \begin{array}{c}
 h^{\prime}  \\
 0  \\
 0 \\
  \end{array}
\right),
\]
where
\[
c_{\sigma}\left( t \right)=c_{i}\left( t \right)u_{i}+c_{j}\left( t \right)u_{j}+c_{k}\left( t \right)u_{k}=\left( c_{i}\left( t \right), c_{j}\left( t \right), c_{k}\left( t \right)\right),
\]
$\left( i, j, k \right)$ is a cyclic permutation of $\left( 1, 2, 3 \right)$ and $( u_{1}, u_{2}, u_{3} )$ is a Milnor basis for $\mathfrak{su}\left( 2 \right)$.
\end{pro}

\begin{proof} 
First observe
\begin{align*}
& c_{\gamma}\left( t \right)=e^{-h\left( t \right)u_{i}}\sigma^{-1}\left( \sigma^{\prime}e^{h\left( t \right)u_{i}}+\sigma h^{\prime}\left( t \right)u_{i}e^{h\left( t \right)u_{i}}\right)
\\
&=\operatorname{Ad}_{e^{h\left( t \right)u_{i}}}\left( c_{\sigma}\left( t \right) \right)+h^{\prime}\left( t \right)u_{i}.
\end{align*}
By Proposition~\ref{p.adjoint} applied to $c_{\sigma}\left( t \right)=c_{1}\left( t \right)u_{1}+c_{2}\left( t \right)u_{2}+c_{3}\left( t \right)u_{3}$

\begin{align*}
& \operatorname{Ad}_{e^{h\left( t \right)u_{i}}}\left( c_{\sigma}\left( t \right)\right)
\\
&=c_{1}\left( t \right)u_{i}+c_{2}\left( t \right)\left( \cos h\left( t \right) u_{j}+\sin h\left( t \right)[u_{j}, u_{i}]\right)+
c_{3}\left( t \right)\left( \cos h\left( t \right) u_{k}+\sin h\left( t \right)[u_{k}, u_{i}]\right)
\\
&=c_{i}\left( t \right)u_{i}+c_{j}\left( t \right)\left( \cos h\left( t \right) u_{j}-\sin h\left( t \right)u_{k}\right)+
c_{k}\left( t \right)\left( \cos h\left( t \right) u_{k}+\sin h\left( t \right)u_{j}\right),
\end{align*}
$\left( i, j, k \right)$ is a cyclic permutation of $\left( 1, 2, 3 \right)$. Therefore

\begin{align*}
& c_{\gamma}\left( t \right)=
\left(c_{i}\left( t \right)+h^{\prime}\left( t \right)\right)u_{i}
+ \left(c_{j}\left( t \right)\cos h\left( t \right) +c_{k}\left( t \right)\sin h\left( t \right)\right)u_{j}
\\
& +
 \left( -c_{j}\left( t \right)\sin h\left( t \right)+c_{k}\left( t \right)\cos h\left( t \right)\right)u_{k}.
\end{align*}
This can be written in a matrix form as
\[
c_{\gamma}\left( t \right)=
\left(
  \begin{array}{ccc}
   1 & 0 & 0 \\
   0 &\cos h & \sin h  \\
 0& -\sin h & \cos h  \\
  \end{array}
\right)
\left(
  \begin{array}{c}
 c_{i}  \\
 c_{j}  \\
 c_{k} \\
  \end{array}
\right)
+\left(
  \begin{array}{c}
 h^{\prime}  \\
 0  \\
 0 \\
  \end{array}
\right).
\]
\end{proof}

\begin{cor}\label{c.4.8} By iterating Proposition~\ref{p.4.7}, for $\gamma\left( t \right)=\Psi_{\operatorname{SU}\left( 2 \right)}\left( x_{1}\left( t \right), x_{2}\left( t \right), x_{3}\left( t \right)\right)$ defined by Definition~\ref{l.2ndKindCoordinates} we have

\begin{align*}
& c_{\gamma}\left( t \right)=
\left(
  \begin{array}{ccc}
 1 & 0 & 0 \\
 0 & \cos x_{1} & \sin x_{1} \\
 0 & -\sin x_{1} & \cos x_{1} \\
  \end{array}
\right)
\left(
  \begin{array}{ccc}
 \cos x_{2} & 0 &  -\sin x_{2} \\
 0 &  1 & 0 \\
  \sin x_{2} & 0 & \cos x_{2} \\
  \end{array}
\right)
\left(
  \begin{array}{c}
   0 \\
 0  \\
   x_{3}^{\prime}  \\
  \end{array}
\right)
\\
& +
\left(
  \begin{array}{ccc}
 1 & 0 & 0 \\
 0 & \cos x_{1} & \sin x_{1} \\
 0 & -\sin x_{1} & \cos x_{1} \\
  \end{array}
\right)
\left(
  \begin{array}{c}
 0 \\
 x_{2}^{\prime}  \\
   0\\
  \end{array}
\right)
+
\left(
  \begin{array}{c}
 x_{1}^{\prime} \\
  0 \\
   0\\
  \end{array}
\right)
\\
&
=
\left(
  \begin{array}{ccc}
 1 & 0 & 0 \\
 0 & \cos x_{1} & \sin x_{1} \\
 0 & -\sin x_{1} & \cos x_{1} \\
  \end{array}
\right)
\left(
  \begin{array}{ccc}
 1 & 0 & -\sin x_{2} \\
 0 & 1 & 0 \\
 0 & 0 & \cos x_{2} \\
  \end{array}
\right)
\left(
  \begin{array}{c}
 x_{1}^{\prime} \\
 x_{2}^{\prime}  \\
 x_{3}^{\prime}\\
  \end{array}
\right).
\end{align*}

In particular, since $u_1, u_2, u_3$ are orthonormal for the
bi-invariant Riemannian metric $g_0$, the Jacobian of $\Psi_{\operatorname{SU}\left( 2 \right)}$ with
respect to $g_0$ is given by $J(x_1, x_2, x_3) = |\cos x_2|$.

Moreover, by inverting the matrix to solve for $(x_1^{\prime}, x_2^{\prime}, x_3^{\prime})$,
we have that the coefficients of the Maurer-Cartan form $c_{\gamma}=\left( \alpha_{1}, \alpha_{2}, \alpha_{3} \right)$ in any standard Milnor basis satisfy

\begin{align}\label{e.MC-ODE}
\left(
  \begin{array}{c}
 x_{1}^{\prime}  \\
 x_{2}^{\prime}  \\
 x_{3}^{\prime} \\
  \end{array}
\right)&=
\left(
  \begin{array}{ccc}
  1 & 0 & \frac{\sin x_{2}}{\cos x_{2}} \\
 0 & 1 & 0 \\
 0 & 0 & \frac{1}{\cos x_{2}} \\
  \end{array}
\right)
\left(
  \begin{array}{ccc}
 1 & 0 & 0 \\
 0 & \cos x_{1} & -\sin x_{1} \\
 0 & \sin x_{1} & \cos x_{1} \\
  \end{array}
\right)
\left(
  \begin{array}{c}
 \alpha_{1}  \\
 \alpha_{2}  \\
 \alpha_{3} \\
  \end{array}
  \right) \\
&= \left(
  \begin{array}{ccc}
 1 & \sin x_1 \tan x_2 & \cos x_1 \tan x_2 \\
 0 & \cos x_{1} & -\sin x_{1} \\
 0 & \sin x_{1} \sec x_2 & \cos x_{1} \sec x_2 \\
  \end{array}
\right)
\left(
  \begin{array}{c}
 \alpha_{1}  \\
 \alpha_{2}  \\
 \alpha_{3} \\
  \end{array}
  \right).
\end{align}
\end{cor}

\subsection{Exponential identities}

A fundamental property of geometry on Lie groups is that for vectors
$v,w \in \mathfrak{g}$, one can move in the $[v,w]$ direction by
combining steps in the $v$ and $w$ directions.  Qualitatively, this is
evident from the Baker--Campbell---Dynkin--Hausdorff formula as follows. 
\begin{equation}
  \exp(s v) \exp(t w) \exp(-sv) \exp(-tw) = \exp\left(st [v,w] + o(s^2+t^2)\right).
\end{equation}
In the case of $\operatorname{SU}(2)$, given a left-invariant metric $g$ with
parameters $a_1, a_2, a_3$ and corresponding standard Milnor basis
$u_1, u_2, u_3$, the relation $[u_1, u_2]=u_3$ gives
\begin{equation}\label{milnor-exp-commute}
  \exp(s u_1) \exp(t u_2) \exp(-s u_1) \exp(-t u_2) \approx \exp(st u_3).
\end{equation}
It is easy to see in general that $d_g(e, \exp(t u_i)) \leqslant a_i t$, so
\eqref{milnor-exp-commute} together with
the left-invariance of $d_g$ then implies
\begin{equation*}
  d_g(e, \exp(st u_3)) \lesssim 2 (s a_1 + t a_2)
\end{equation*}
which, depending on the values of $s, t, a_i$ may be a much better
bound than the trivial $d_g(e, \exp(st u_3)) \leqslant st a_3$.  However,
this bound is useful only when $s,t$ are not too small, and in that
case one must pay careful attention to the remainder terms in the
expansion, and obtain uniform control over its $u_1, u_2, u_3$
components separately.  Moreover, for our current purposes in
$\operatorname{SU}(2) \times \mathbb{R}^3$, it is also necessary to iterate
\eqref{milnor-exp-commute} to take advantage of third-order brackets
such as $[[v_1, v_2], v_1] = u_2$.  This approach is possible but
tedious, and so we instead make use of a helpful closed-form identity
which eliminates most of the unwanted remainder terms.

\begin{theo}\label{long-identity}
  Let $(u_1, u_2, u_3)$ be a standard Milnor basis for
  $\mathfrak{su}(2)$.  For any $s,t \in \mathbb{R}$, let $\tau =
  \tau(s,t) = \arctan\left(\cos \frac{s}{2} \tan \frac{t}{2}\right)$.
  Then
  \begin{equation}\label{identity-f-def}
 e^{-\tau u_2} e^{s u_1/2} e^{t u_2} e^{-s u_1} e^{-t u_2} e^{s
   u_1/2} e^{\tau u_2} = e^{f(s,t) u_3}
  \end{equation}
  where
  \begin{equation}\label{f-formula}
 f(s,t) =2 \arccos\left(\cos^2 \frac{s}{2} + \sin^2 \frac{s}{2} \cos t\right).
  \end{equation}
  Moreover, we have $|\tau(s,t)| \leqslant \frac{|t|}{2}$, and
  for some small constants $c, \eta >0$, we have $|f(s,t)| \geqslant c |st|$ for all
  $|s|,|t| \leqslant \eta$.
\end{theo}

\begin{proof}
More generally, suppose $X, Y, Z$ are matrix Lie algebra elements satisfying \eqref{XYZ-relations}. 
Then we claim that
\begin{align*}
& e^{-\tau Y} e^{uX} e^{tY} e^{-sX} e^{-tY}e^{uX}e^{\tau Y}
\\
& =
\sqrt{a^{2}+c^{2}} Z
+\left( \cos\frac{s}{2} \cos u + \sin\frac{s}{2} \cos t \sin u \right) I,
\end{align*}
where
\begin{align*}
& a=2\left( - \sin\frac{s}{2} \cos t \cos u + \cos\frac{s}{2}\sin u\right),
\\
& c= 2\sin\frac{s}{2}\sin t
\\
& \tau=\arctan \frac{a}{c}.
\end{align*}

First, note that the assumed identities \eqref{XYZ-relations} imply that

\begin{equation*}
  XY+YX=0, \quad  ZX+XZ=0, \quad  YZ+ZY=0.
\end{equation*}
This implies that $e^{uX}Y e^{uX}=Y$ and $e^{uX}Z e^{uX}=Z$ since the derivatives in $u$ of these functions are $0$. Then by Rodrigues' formula \eqref{e.Rodrigues1} we have that for any $a, b, c, d \in \mathbb{R}$
\begin{multline} \label{e.6.4}
 e^{uX}\left( aX+bY+cZ+dI \right) e^{uX}= aXe^{2uX} +bY+cZ+de^{2uX}
\\
= \left( a \cos u  +2d\sin u\right) X +bY+cZ+\left( d\cos u -\frac{a}{2} \sin u\right) I.
\end{multline}
Note that this procedure does not change $Y$ and $Z$ components. Now we can use Rodrigues' formula \eqref{e.Rodrigues1} and Corollary~\ref{cor.5.4}  to see that
\begin{align*}
& e^{tY} e^{-sX} e^{-tY}
\\
& =\exp\left(  -s \cdot \cos t X +s \cdot \sin t Z  \right)=
\cos\frac{s}{2}I +\frac{\sin\frac{s}{2}}{\frac{s}{2}}\left(  -s \cdot \cos t X +s \cdot \sin t Z  \right)
\\
& =\cos\frac{s}{2}I   - 2\sin\frac{s}{2} \cos t \cdot X +  2\sin\frac{s}{2}\sin t \cdot Z.
\end{align*}
Apply Equation \eqref{e.6.4} with $a=- 2\sin\frac{s}{2} \cos t, b=0, c=2\sin\frac{s}{2}\sin t, d=\cos\frac{s}{2}$ to see that
\begin{align*}
e^{uX} e^{tY} e^{-sX} e^{-tY}e^{uX} &=
2\left( - \sin\frac{s}{2} \cos t \cos u + \cos\frac{s}{2}\sin u\right)
X +2\sin\frac{s}{2}\sin t \cdot Z \\
&\quad +\left( \cos\frac{s}{2} \cos u + \sin\frac{s}{2} \cos t \sin u\right) I.
\end{align*}
By Proposition~\ref{p.adjoint} and Rodrigues' formula \eqref{e.Rodrigues1} we see that for any $a, b, c, d \in \mathbb{R}$ we have
\begin{align*}
& e^{-\tau Y}\left( aX+bY+cZ+dI \right) e^{\tau Y}
\\
& =\left( \cos \frac{\tau}{2} I-2\sin \frac{\tau}{2} Y \right)\left( aX+bY+cZ+dI \right) \left( \cos \frac{\tau}{2} I+\sin \frac{\tau}{2} Y \right)
\\
& =
\left( a\cos \tau -c\sin \tau\right) X
+b Y
+\left( a\sin \tau+ c\cos \tau \right) Z
+d I,
\end{align*}
This procedure does not change the $Y$ component. Now we use the formula with
\begin{align*}
& a=2\left( - \sin\frac{s}{2} \cos t \cos u + \cos\frac{s}{2}\sin u\right),
\\
& b=0,
\\
& c= 2\sin\frac{s}{2}\sin t,
\\
& d=\cos\frac{s}{2} \cos u + \sin\frac{s}{2} \cos t \sin u,
\end{align*}
therefore
\begin{align*}
& e^{-\tau Y} e^{uX} e^{tY} e^{-sX} e^{-tY}e^{uX}e^{\tau Y}
\\
& =\left( \cos \frac{\tau}{2} I-2\sin \frac{\tau}{2} Y \right)\left( aX+bY+cZ+dI \right) \left( \cos \frac{\tau}{2} I+2\sin \frac{\tau}{2} Y \right)
\\
& =
\left( a\cos \tau -c\sin \tau\right) X
+\left( a\sin \tau+ c\cos \tau \right) Z
+\left( \cos\frac{s}{2} \cos u + \sin\frac{s}{2} \cos t \sin u \right) I.
\end{align*}
If we choose $\tau$ such that $\tan \tau = \frac{a}{c}$, then the
coefficient of $X$ vanishes.  For such $\tau$ we have $a \sin \tau + c
\cos \tau = \sqrt{a^2 + c^2}$, and so
\begin{align*}
& e^{-\tau Y} e^{uX} e^{tY} e^{-sX} e^{-tY}e^{uX}e^{\tau Y}
\\
& =
\sqrt{a^{2}+c^{2}} Z
+\left( \cos\frac{s}{2} \cos u + \sin\frac{s}{2} \cos t \sin u \right) I.
\end{align*}
In particular, if $u=s/2$, then
\begin{equation}\label{XYZ-final}
 e^{-\tau Y} e^{uX} e^{tY} e^{-sX} e^{-tY}e^{uX}e^{\tau Y}
 =
\sqrt{a^{2}+c^{2}} Z
+\left( \cos^{2}\frac{s}{2}  + \sin^{2}\frac{s}{2} \cos t  \right) I,
\end{equation}
where
\begin{align*}
& a=2\left( - \sin\frac{s}{2} \cos \frac{s}{2} \cos t + \sin \frac{s}{2}\cos\frac{s}{2}\right)=\sin s\left( 1-\cos t\right),
\\
& c= 2\sin\frac{s}{2}\sin t
\\
& \tau=\arctan \left( \cos\frac{s}{2} \tan \frac{t}{2} \right).
\end{align*}
Comparing coefficients of $I$ between \eqref{XYZ-final} and \eqref{e.Rodrigues1} recovers the formula \eqref{f-formula} for $f(s,t)$.

The claimed inequality $|\tau(s,t)| \leqslant \frac{|t|}{2}$ is immediate
from the monotonicity of $\arctan$ and the fact that $\left|\cos
\frac{s}{2}\right| \leqslant 1$.  As for the lower bound on $f$, we may use
the Pythagorean identity to rewrite
\eqref{f-formula} as
\begin{equation}
  \cos \frac{f(s,t)}{2} = 1 - \left(\sin^2\frac{s}{2}\right) \left(1 - \cos t\right).
\end{equation}
Using the elementary inequality $\cos \theta \geqslant 1 -
\frac{\theta^2}{2}$, and rearranging, we have
\begin{equation*}
  |f(s,t)| \geqslant \sqrt{8 \left(\sin^2 \frac{s}{2}\right) (1 - \cos t)}.
\end{equation*}
Now using the crude bounds $|\sin \frac{s}{2}| \geqslant \frac{|s|}{4}$ and $1 - \cos t \geqslant \frac{t^2}{4}$, both valid for sufficiently small $|s|, |t|$ (for instance, when $|s| \leqslant \pi$ and $|t| \leqslant \frac{\pi}{2}$), we can conclude that $|f(s,t)| \geqslant |st|/\sqrt{8}$ for
all such $s,t$.
\end{proof}

An important feature of \eqref{identity-f-def} is that on the left-hand side, the coefficients of $u_1$ and of $u_2$ each sum to zero.  Hence we immediately have the following extension.

\begin{cor}\label{identity-with-v}
  Let $G$ be any Lie group containing $\mathrm{SU}(2)$ as a subgroup, and let $u_1, u_2, u_3$ be a standard Milnor basis for $\mathfrak{su}(2) \subset \mathfrak{g} = \operatorname{Lie}(G)$.  Let $f_1, f_2 \in \mathfrak{g}$ be central elements, and let $v_i = u_i + f_i$ for $i=1,2$.  Then \eqref{identity-f-def} holds with $v_1, v_2$ in place of $u_1, u_2$.
\end{cor}

In particular, this will apply when $G = \mathrm{SU}(2) \times \mathbb{R}^3$ and the $v_i$ are a decoupled Milnor basis.

\section{Coordinates}\label{coordinates-sec}

For the remainder of the paper, $g$ denotes a decoupled left-invariant Riemannian metric on $G = \operatorname{SU}(2) \times \mathbb{R}^3$ with decoupled Milnor basis $(v_1, v_2, v_3, f_1, f_2, f_3)$ where $v_i = u_i + d f_i$ as defined in  Definition~\ref{decoupled-def}, and parameters $(a_1, a_2, a_3, d)$ as in Notation~\ref{parameters}.

Recall from Proposition~\ref{up-to-isometry} that a decoupled metric is uniquely defined up to isometric isomorphism by its parameters.  As such, all the computations below depend on the parameters of the metric $g$, and not on its decoupled Milnor basis $(v_i, f_i)$.  For instance, as noted in Remark~\ref{no-basis-dependence}, we could suppose without loss of generality that $(u_1, u_2, u_3)$ are always the standard Pauli matrices in $\mathfrak{su}(2)$ and $(f_1, f_2, f_3)$ is the standard Euclidean basis for $\mathbb{R}^3$.  Thus in the following computations, whenever a constant is defined in terms of the vectors $u_i, f_i$ of a  decoupled Milnor basis, without reference to the parameters of the metric, then the constant is in fact universal and does not depend on the metric $g$.

In what follows, we use $c, C$ to denote universal constants, where $c > 0$ is small and $C < \infty$ is large.  To simplify notation, their values may change from line to line.  

\begin{notation}
In our chosen normalization, for each vector $u_i$ of a standard Milnor basis for $\mathfrak{su}(2)$, the map $x
\mapsto e^{x u_i}$ is periodic with period $4\pi$. This follows for instance from Rodrigues' formula \eqref{e.Rodrigues1}.  As such, we let $\mathbb{S}$ denote a circle of circumference $4\pi$,
viewed as the quotient of $\mathbb{R}$ by the lattice $\{\pm 2 \pi, \pm 6 \pi, \pm 10\pi, \dots\}$, and then the map $\mathbb{S} \ni x \mapsto e^{x u_i}$ is well-defined.
\end{notation}

\begin{notation}\label{Psi-def}
Extending the coordinates of the second kind $\Psi_{\operatorname{SU}(2)}$ defined in Definition~\ref{l.2ndKindCoordinates} above, define $\Psi:  (\mathbb{S} \times \mathbb{R})^3 \to \operatorname{SU}(2) \times \mathbb{R}^3$ by
\begin{align*}\label{psi-def-eq}
\Psi(x_1, y_1, x_2, y_2, x_3, y_3) &= e^{x_3 u_3} e^{y_3 f_3}
 e^{x_2 u_2} e^{y_2 f_2} e^{x_1 u_1} e^{y_1 f_1} \\
 &= \Psi_{\operatorname{SU}(2)}(x_1, x_2, x_3) e^{y_1 f_1} e^{y_2 f_2} e^{y_3 f_3} \\
 &= \Psi_{\operatorname{SU}(2)}(x_1, x_2, x_3) e^{y_1 f_1 + y_2 f_2 + y_3 f_3}
\end{align*}
recalling that all the $f_i$ are in the center of $\mathfrak{g}$.
\end{notation}

\begin{notation}
We denote by $J : (\mathbb{S} \times \mathbb{R})^3 \to [0, \infty)$ the absolute value of the Jacobian determinant of $\Psi$ with respect to the Euclidean metric on $(\mathbb{S} \times  \mathbb{R})^3$ and the bi-invariant reference metric $g_0$ on $\operatorname{SU}(2) \times \mathbb{R}^3$ defined in Notation~\ref{g0-def}).  As the volume measures of these metrics are respectively Lebesgue measure and the normalized Haar measure $\mu_0$, then for any Borel set $K \subset (\mathbb{S} \times  \mathbb{R})^3$, we have
\begin{equation*}
\mu_0(\Psi(K)) \leqslant \int_K J\,dm
\end{equation*}
with equality if $\Psi$ is injective on $K$.
\end{notation}

By the translation invariance of $\mu_0$, we see that $J$ does not depend on the $y_i$ variables, so as noted in Corollary \ref{c.4.8}, we have $J = J(x_2) = |\cos x_2|$.  In particular, $|J| \leqslant 1$, so for any $K \subset (\mathbb{S} \times \mathbb{R})^3$ we have $\mu_0(\Psi(K)) \leqslant |K|$.     We also note that $J(0) = 1$.

\begin{notation}\label{iota}
  Let $\iota > 0$ be a small enough number so that on the set
  \begin{equation*}
    A_{\iota} := \{ |x_1|, |x_2|, |x_3| < \iota \} \subset (\mathbb{S} \times \mathbb{R})^3,
  \end{equation*}
  the map $\Psi$ is injective and has $J \geqslant \frac{1}{2}$.    The existence of such $\iota$ is clear from the inverse function theorem, or from Proposition \ref{p.5.4}.  Then for $K \subset A_\iota$ we have $\mu_0(\Psi(K)) \geqslant \frac{1}{2} |K|$.
  \end{notation}

The key property of the geometry of $\operatorname{SU}(2) \times \mathbb{R}^3$, in terms of the decoupled Milnor basis $(v_i, f_i)$, is that the metric $g$ tells us the \emph{cost} of moving in the $v_i$ and
$f_i$ directions respectively.  But we can also move in the $u_i$ direction by exploiting the bracket relations $u_i = [u_j, u_k] = -[[u_i, u_j], u_j] = [[u_i, u_k], u_k]$; for instance, $e^{su_j}
e^{tu_k} e^{-su_j} e^{-t u_k} \approx e^{st u_i}$ plus lower order terms.  The directions $u_i, v_i, f_i$ are linearly dependent, so we need to keep track of all the points
that can be reached by any combination of those motions.

Any linear combination of $u_i, v_i, f_i$ may be written in terms of $u_i, f_i$ alone via
\begin{equation*}
\mu u_i + \nu v_i + \xi f_i = (\mu + \nu)u_i + (d \nu + \xi) f_i,
  \qquad \mu, \nu, \xi \in \mathbb{R}
\end{equation*}
and so if we constrain the sizes of the coefficients $\mu, \nu, \xi$ we obtain the following set.

\begin{notation}\label{H-notation}
Given parameters $\mu^{\ast}, \nu^{\ast}, \xi^{\ast}, d \geqslant 0$, let
  \begin{equation*}
 H = H(\mu^{\ast}, \nu^{\ast}, \xi^{\ast}, d) = \{ (\mu + \nu, d \nu + \xi) \in
 \mathbb{S} \times \mathbb{R} : |\mu| \leqslant \mu^{\ast}, |\nu| \leqslant \nu^{\ast},
 |\xi| \leqslant \xi^{\ast}\}.
  \end{equation*}

We also define the corresponding \emph{lifted} version in the plane.
  \begin{equation*}
 \widetilde{H} = \{ (\mu + \nu, d \nu + \xi) \in \mathbb{R}^2 : |\mu| \leqslant \mu^{\ast}, |\nu| \leqslant \nu^{\ast},
 |\xi| \leqslant \xi^{\ast}\}
  \end{equation*}
so that $H = \pi(\widetilde{H})$, where $\pi : \mathbb{R}^2 \to \mathbb{S}\times \mathbb{R}$ is the quotient map.

Finally, we let $H_\iota = H \cap ([-\iota, \iota] \times \mathbb{R})$, with $\iota$ as in Notation \ref{iota}.
\end{notation}

The set $\widetilde{H} \subset \mathbb{R}^2$ is an irregular hexagon formed by sweeping the rectangle $[-\mu^{\ast}, \mu^{\ast}] \times [-\xi^{\ast}, \xi^{\ast}]$ along the diagonal line of slope $d$ between the points $\pm(\nu^{\ast}, d\nu^{\ast})$; it can also be seen as the projection of a three-dimensional cuboid into the plane.  See Figure~\ref{hexagon-fig}.  The set $H \subset \mathbb{S}
\times \mathbb{R}$ is then  the hexagon $\widetilde{H}$ wrapped horizontally around a cylinder, where it may or may not overlap itself.  See Figure~\ref{hex-cyl-fig} for examples.  Finally, $H_\iota$ is the portion of $H$ lying in a narrow vertical strip on the cylinder, represented in Figure~\ref{hex-cyl-fig} by dashed lines.

\begin{figure}
  \includegraphics{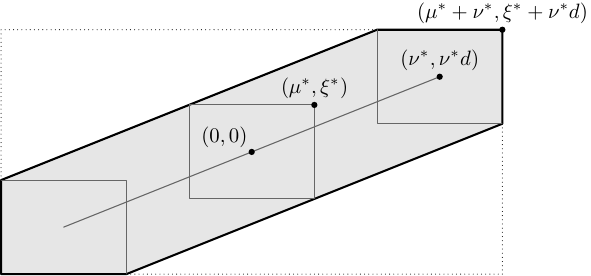}
  \caption{The hexagon $\widetilde{H}(\mu^{\ast}, \nu^{\ast}, \xi^{\ast})$}\label{hexagon-fig}
\end{figure}

\begin{figure}
  \begin{center}
  \begin{tabular}{ccc}
    \includegraphics{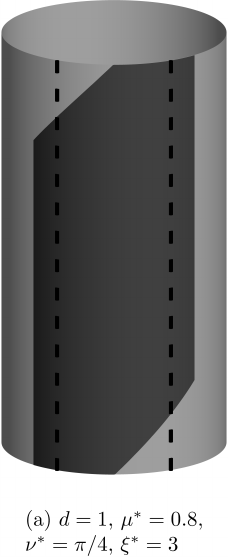} &
    \includegraphics{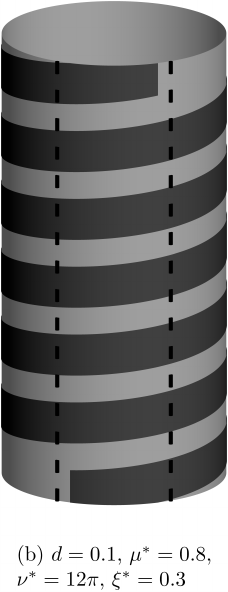} &
    \includegraphics{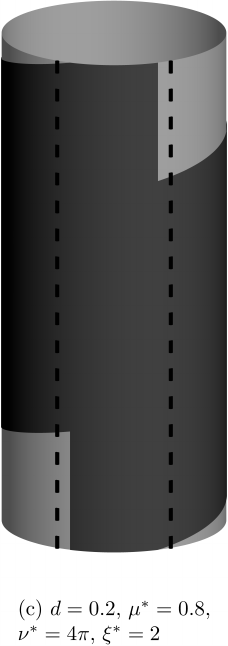}
  \end{tabular}
  \end{center}
  \caption{The set $H$ for various values of the parameters $d, \mu^\ast, \nu^\ast, \xi^\ast$, in dark gray. The cylinder $\mathbb{S} \times \mathbb{R}$ is shown in light gray for reference.  The portion of $H$ between the vertical dashed lines represents $H_\iota$.  Note that in (a) and (b), the projection $\pi : \widetilde{H} \to H$ is injective, but in (c) it is not.} \label{hex-cyl-fig}
  \end{figure}

\begin{lem}\label{H-volume-estimate}
Let
  \begin{equation}\label{VH-def}
 \overline{V}_H(\mu^{\ast}, \nu^{\ast}, \xi^{\ast}, d) = \min( d \mu^{\ast} \nu^{\ast} +
 \mu^{\ast} \xi^{\ast} + \nu^{\ast} \xi^{\ast}, d \nu^{\ast} + \xi^{\ast}).
  \end{equation}
Then for some universal constants $0 < c < C < \infty$, we have
  \begin{equation}\label{H-volume-eqn}
 c \overline{V}_H \leqslant |H| \leqslant C \overline{V}_H.
  \end{equation}
\end{lem}

\begin{proof}
  First, if $d = 0$ then $\widetilde{H}$ is simply a rectangle of width $2 (\mu^\ast + \nu^\ast)$ and height $2 \xi^\ast$, so we have $|H| = 2 \xi^\ast \min(2(\mu^\ast + \nu^\ast), 4 \pi) \asymp \overline{V}_H$.  Assume then that $d > 0$.
  
We begin with the upper bounds.  On the one hand, since $H$ is formed by wrapping $\widetilde{H}$ around the cylinder via the local isometry $\pi$, we have $|H| \leqslant |\widetilde{H}|$.  The latter can be computed explicitly with simple geometry, giving
\begin{equation}\label{nowrap-upper}
 |H| \leqslant |\widetilde{H}| = 4(d \mu^{\ast} \nu^{\ast} + \mu^{\ast} \xi^{\ast} +
 \nu^{\ast} \xi^{\ast}).
\end{equation}
On the other hand, as the overall height of $H$ is $2(d \nu^{\ast} + \xi^{\ast})$, it is contained in the truncated cylinder $\mathbb{S} \times [-(d \nu^{\ast} + \xi), d \nu^{\ast} + \xi]$. As $\mathbb{S}$ has circumference $4 \pi$, this gives
\begin{equation}\label{wrap-upper}
 |H| \leqslant 8 \pi (d \nu^{\ast} + \xi^{\ast}).
\end{equation}
Combining \eqref{nowrap-upper} and \eqref{wrap-upper} yields the upper bound of \eqref{H-volume-eqn}.

Now we turn to the lower bounds.  From Figure \ref{hexagon-fig}, we see that $\widetilde{H}$ is bounded by the six lines $x = \pm (\mu^\ast + \nu^\ast)$, $y= \pm (\xi^\ast +  d \nu^\ast)$, $y-dx=\pm(d \nu^\ast + \xi^\ast)$.  Thus the cross-sectional width of $\widetilde{H}$ at a particular value of $|y| \leqslant 2(\xi^\ast + d \nu^\ast)$ is given by
\begin{equation*}
  w(y) = \min\left(2(\mu^\ast + \nu^\ast), 2\left(\mu^\ast + \frac{\xi^\ast}{d}\right), \left(2 \mu^\ast + \nu^\ast + \frac{\xi^\ast - |y|}{d}\right)^+\right).
\end{equation*}
In particular, if $\mu^\ast + \nu^\ast \leqslant 2\pi$ or $\mu^\ast + \xi^\ast/d \leqslant 2\pi$, the maximum cross-sectional width of $\widetilde{H}$ is less than $4\pi$, and so $\pi$ maps $\widetilde{H}$ injectively onto $H$.  In this case we have $|H| = |\widetilde{H}| \geqslant \overline{V}_H$.

Suppose instead that we have $\mu^\ast + \nu^\ast \geqslant 2\pi$ and $\mu^\ast + \xi^\ast/d \geqslant 2\pi$.  Let $0 < \epsilon < 4\pi$ be a sufficiently small constant, to be chosen later.
With a little algebra, we see that we have $w(y) \geqslant \epsilon$ provided that
\begin{align*}
  |y| &\leqslant \xi^\ast + d(2 \mu^\ast + \nu^\ast - \epsilon).
\end{align*}
If $\mu^\ast \geqslant \epsilon$ then $w(y) \geqslant \epsilon$ for all $|y| \leqslant 2(\xi^\ast + d \nu^\ast)$, so
\begin{equation*}
  |H| \geqslant 2 \epsilon (\xi^\ast + d \nu^\ast) \geqslant 2 \epsilon \overline{V}_H.
\end{equation*}
On the other hand, if $\mu^\ast \leqslant \epsilon$ then we have $\nu^\ast \geqslant 2\pi-\epsilon$.  Choose $\epsilon$ so small that $2 \pi - \epsilon \geqslant 2 \epsilon$; then we have $\nu^\ast - \epsilon \geqslant \nu^\ast/2$, and so $w(y) \geqslant \epsilon$ whenever $|y| \leqslant \xi^\ast + d \nu^\ast / 2$.  Hence in this case,
\begin{equation*}
  |H| \geqslant \epsilon \cdot \left(\xi^\ast + d \nu^\ast / 2\right) \geqslant \frac{\epsilon}{2} \overline{V}_H.
\end{equation*}
\end{proof}

As for $H_\iota$, the following symmetry argument shows that it contains \emph{enough} of $H$.

\begin{lem}\label{H-iota-lower}
There is a constant $c$ such that $|H_\iota| \geqslant c |H|$ for all values of $d, \mu^\ast, \nu^\ast, \xi^\ast$.  Thus $|H_\iota| \geqslant c^{\prime} \overline{V}_H$ for some other  constant $c^{\prime}$.
\end{lem}

\begin{proof}
Let $M \geqslant 4 \pi / \iota$ be an integer, so that the circle $\mathbb{S}$ can be partitioned into $M$ equal intervals, each of width $4 \pi / M \leqslant \iota$.

Suppose $(x,y) \in H$, so that $x = \mu + \nu$, $y = d \nu + \xi$ for some $|\mu| \leqslant \mu^{\ast}$, $|\nu| \leqslant \nu^{\ast}$, $|\xi| \leqslant \xi^{\ast}$. We can write
\begin{align*}
 \mu &= 4 \pi j_\mu + k_\mu \iota  + \mu^{\prime} \\
 \nu &= 4 \pi j_\nu + k_\nu \iota  + \nu^{\prime},
\end{align*}
where $j_\mu, j_\nu, k_\mu, k_\nu$ are integers, $|k_\mu|, |k_\nu| \leqslant M$, and $|\mu^{\prime}|, |\nu^{\prime}| \leqslant \iota/2$.  Also, by adjusting $j_\mu, j_\nu$ by $\pm 1$ if needed, we may ensure that $k_\mu, k_\nu$ are of the same signs as $\mu, \nu$ respectively.   Set $\bar{\mu} = \mu - k_\mu \iota$, $\bar{\nu} = \nu - k_\nu \iota$.  By the sign assumptions on $k_\mu, k_\nu$, we have $|\bar{\mu}| \leqslant |\mu| \leqslant \mu^{\ast}$ and likewise $|\bar{\nu}| \leqslant \nu^{\ast}$.  Hence $(\bar{x}, \bar{y}) := (\bar{\mu} + \bar{\nu}, d \bar{\nu} + \xi) \in H$.  Also, $\bar{x} = 4 \pi (j_\mu + j_\nu) + (\mu^{\prime} + \nu^{\prime})$, where $|\mu^{\prime} + \nu^{\prime}| \leqslant \iota$, so $(\bar{x}, \bar{y}) \in H_\iota$.

Thus, we have written
\begin{equation*}
(x, y) = (\bar{x}, \bar{y}) + ( (k_\mu + k_\nu) \iota, d k_\nu \iota),
\end{equation*}
where $|k_\mu|, |k_\nu| \leqslant M$ are integers.  Since $(x,y) \in H$ was arbitrary, we have shown
\begin{equation*}
 H \subseteq \bigcup_{k_\mu, k_\nu = -M}^M H_\iota + ( (k_\mu + k_\nu) \iota, d k_\nu \iota).
\end{equation*}
So $H$ is contained in the union of $(2M+1)^2$ translated copies of $H_\iota$, and thus $|H| \leqslant (2M+1)^2 |H_\iota|$.  This is the desired  statement when we take $c = 1/(2M+1)^2$.
\end{proof}

\begin{cor}\label{K-estimate}
  For any $\mu_i^\ast, \nu_i^\ast, \xi_i^\ast$, $i=1,2,3$, let
  \begin{equation*}
    K = \prod_{i=1}^3 H(\mu_i^{\ast}, \nu_i^{\ast}, \xi_i^{\ast}).
  \end{equation*}
  Then we have
  \begin{equation}\label{PsiK-asymp}
    \mu_0(\Psi(K)) \asymp \prod_{i=1}^3 \overline{V}_H(\mu_i^\ast, \nu_i^\ast, \xi_i^\ast)
  \end{equation}
  where the implicit constants are universal.
\end{cor}

\begin{proof}
  For the upper bound, since the Jacobian $J$ of $\Psi$ is bounded by $1$, we have
  \begin{align*}
    \mu_0(\Psi(K)) &\leqslant |K| 
    = \prod_{i=1}^3 |H(\mu_i^{\ast}, \nu_i^{\ast}, \xi_i^{\ast})| 
    \leqslant C^3 \prod_{i=1}^3  \overline{V}_H(\mu_i^\ast, \nu_i^\ast, \xi_i^\ast).
  \end{align*}
  For the lower bound, let
  \begin{equation*}
    K_\iota = \prod_{i=1}^3 H_\iota(\mu_i^{\ast}, \nu_i^{\ast}, \xi_i^{\ast})
  \end{equation*}
  so that $K_\iota \subset K$.  Then $K_\iota \subset A_\iota$ (see Notation \ref{iota}), so that $J \geqslant \frac{1}{2}$ on $K_\iota$.  Hence
  \begin{align*}
    \mu_0(\Psi(K)) &\geqslant \mu_0(\Psi(K_\iota)) \\
    &\geqslant \frac{1}{2} |K_\iota| \\
    &= \frac{1}{2} \prod_{i=1}^3 |H_\iota(\mu_i^{\ast}, \nu_i^{\ast}, \xi_i^{\ast})| \\
    &\geqslant \frac{1}{2}c^3 \prod_{i=1}^3  \overline{V}_H(\mu_i^\ast, \nu_i^\ast, \xi_i^\ast).
  \end{align*}
\end{proof}

\begin{rem}
  Referring back to the notation of Theorem~\ref{volume-estimate}, we will eventually take $\mu_i^\ast = \rho_i$, $\nu_i^\ast = r/a_i$, $\xi_i^\ast = r$.  Then the right side of \eqref{PsiK-asymp} becomes $\overline{V}_g(r)$.
\end{rem}

\begin{rem}
  To obtain the uniform doubling result of Corollary~\ref{decoupled-uniformly-doubling} by itself, without the more detailed volume estimates of   Theorem~\ref{volume-estimate}, it would suffice to show that $|H|$ is a uniformly doubling function of $\mu_i^{\ast}, \nu_i^{\ast}, \xi_i^{\ast}$.  This can be done without the detailed estimates in Lemma~\ref{H-volume-estimate}, by simply applying Lemma~\ref{guivarch} to the projection $\pi : \mathbb{R}^2 \to \mathbb{S} \times \mathbb{R}$ and taking $A = \widetilde{H}$, $B = 2 \widetilde{H}$.
\end{rem}

\section{Lower bounds}\label{lower-sec}

The identity \eqref{identity-f-def} provides an explicit recipe for producing a group element $e^{x u_k}$ as a product of elements of the forms $e^{s u_i}$ and $e^{t u_j}$, where $(i,j,k)$ is any cyclic permutation of $(1,2,3)$.  This helps us describe explicitly a
set of group elements contained in a given ball $B_g(r)$, and thus to bound its volume from below.

\begin{lem}\label{identity-one-bracket}
There are universal constants $\eta > 0$, $M > 0$ such that the following holds.  Let $g$ be any decoupled left-invariant Riemannian metric on $G = \operatorname{SU}(2)$.  Following our usual notation, let $(v_1, v_2, v_3, f_1, f_2, f_3)$ be a decoupled Milnor basis for $g$, let $(a_1, a_2, a_3, d)$ be its parameters (see Definition \ref{parameters}), and let $u_i = v_i - d f_i$, so that $(u_1, u_2, u_3)$ is a standard Milnor basis.   Finally, let $r  > 0$, and set $m_i = \min(r/a_i, \eta)$ as in Notation \ref{m-rho-def}.  Then for all $|\sigma| \leqslant  m_j m_k$, we have $e^{\sigma u_i} \in B_g(Mr)$.  
\end{lem}

\begin{proof}
Notice that for all $|t| \leqslant r/a_i$ we have $e^{t v_i} \in B_g(r)$.  Indeed, the Maurer--Cartan form of the path $\gamma(s) = e^{s v_i}$, $0 \leqslant s \leqslant t$, is simply $v_i$, and $\sqrt{\langle v_i, v_i \rangle_{g}} = a_i$.

To simplify the notation, we will take $(i,j,k) = (3,1,2)$.  Let $f(s,t)$ be as defined in \eqref{f-formula}.  As shown in Theorem \ref{long-identity}, we can  choose $c, \eta$ such that $f(s,t) \geqslant cst$ for all $0 \leqslant s,t \leqslant \eta$.  

Let us suppose first that $|\sigma| \leqslant c m_1 m_2$.  We can suppose without loss of generality that $\sigma \geqslant 0$, since $B_g(Mr)$ is symmetric with respect to the group inverse.  Then we have
\begin{equation*}
0 \leqslant \sigma \leqslant c m_1 m_2 \leqslant f(m_1, m_2)
\end{equation*}
and so by continuity there exist $0 \leqslant s \leqslant m_1$ and $0 \leqslant t \leqslant m_2$ such that $f(s,t) = \sigma$.  By Corollary \ref{identity-with-v} (replacing $f_1, f_2$ by $df_1, df_2$), we have
\begin{equation*}
    e^{\sigma u_3} = e^{f(s,t) u_3} = e^{-\tau v_2} e^{s v_1/2} e^{t v_2} e^{-s v_1} e^{-t v_2} e^{s
   v_1/2} e^{\tau v_2}
\end{equation*}
where $|\tau| \leqslant t/2$.  Since $s \leqslant r/a_1$ and $t \leqslant r/a_2$, each of the seven factors on the right-hand side is an element of $B_g(r)$.  As such, by the left-invariance of the metric $d_g$ and the triangle inequality, we have $e^{\sigma u_3} \in B_{g}(7r)$.

If we only have $|\sigma| \leqslant m_1 m_2$, then by choosing any integer $N \geqslant 1/c$, we can write $e^{\sigma u_3} = (e^{\sigma u_3/N})^N$, and conclude that $e^{\sigma u_3} \in B(7Nr)$.  This is the desired result when we set $M = 7N$.
\end{proof}

\begin{lem}\label{identity-two-brackets}
There are universal constants $\eta > 0$, $M' > 0$ such that the following holds.  Let $g$ be any decoupled left-invariant Riemannian metric on $G = \operatorname{SU}(2) \times \mathbb{R}^3$, with decoupled basis and parameters denoted as before.  Let $r  > 0$, and define $\rho_i = \rho_i(r, a_i, a_2, a_3, \eta)$ as in Notation \ref{m-rho-def}.  Then for all $|\sigma| \leqslant \rho_i$, we have $e^{\sigma u_i} \in B_g(e, M'r)$.  
\end{lem}

\begin{proof}
Again we take $(i,j,k)=(3,1,2)$.  Let $\eta$ be small enough to satisfy both Theorem \ref{long-identity} and Lemma \ref{identity-one-bracket}, and let $c, M$ respectively be as given by those same statements.   As before, we assume that $\sigma \geqslant 0$.

Let us first suppose that $\sigma \leqslant c m_1^2 m_3$.  Assuming without loss of generality that $\eta \leqslant 1$, we have $m_1 m_3 \leqslant \eta^2 \leqslant \eta$, so that
\[
    0 \leqslant \sigma \leqslant c m_1^2 m_3 \leqslant f(m_1, m_1 m_3)
\]
and hence there exist $0 \leqslant s \leqslant m_1$, $0 \leqslant t \leqslant m_1 m_3$ such that $\sigma = f(s,t)$.  Thus by Corollary \ref{identity-with-v} (replacing $f_1$ with $d f_1$ and $f_2$ with $0$) we have
\begin{equation*}
    e^{\sigma u_3} = e^{-\tau u_2} e^{s v_1/2} e^{t u_2} e^{-s v_1} e^{-t u_2} e^{s v_1/2} e^{\tau u_2}
\end{equation*}
where $|\tau| \leqslant t/2$.  We clearly have $e^{s u_1/2}, e^{-s u_1} \in B_g(r)$, and by Lemma \ref{identity-one-bracket}, we have $e^{\pm t u_2}, e^{\pm \tau u_2} \in B_g(Mr)$.  Hence $e^{\sigma u_3} \in B((4M+3)r)$.

If we only have $\sigma \leqslant m_1^2 m_3$, then as before, choosing an integer $N$ with $N \geqslant 1/c$, we have $e^{\sigma u_3} = (e^{\sigma u_3/N})^N \in B_g((4M+3)Nr)$.

If $\sigma \leqslant m_2^2 m_3$, we get the same conclusion by the same argument.

If $\sigma \leqslant m_1 m_2$, then we already showed in Lemma \ref{identity-one-bracket} that $e^{\sigma u_3} \in B_g(Mr)$.

So if $\sigma \leqslant \rho_3 = m_1 m_2 + m_1^2 m_3 + m_2^2 m_3$, then by writing $\sigma$ as a sum of three numbers bounded by $m_1 m_2$, $m_1^2 m_3$, $m_2^2 m_3$ respectively, we have $e^{\sigma u_3}$ as the product of three elements which are in $B_g(Mr)$, $B_g((4M+3)Nr)$, $B_g((4M+3)Nr)$ respectively.  Hence $e^{\sigma u_3} \in B_g(M' r)$ with $M' = M + 2(4M+3)N$.
\end{proof}
  
\begin{theo}\label{ball-contain-inner}
There exist universal constants $\eta, M^{\prime\prime}$ such that the following holds for every decoupled left-invariant Riemannian metric $g$ on $\operatorname{SU}(2) \times \mathbb{R}^3$, with decoupled basis and parameters denoted as before, and every $r > 0$.  Define  $m_i, \rho_i$ in terms of $r, a_i, \eta$ as in Notation~\ref{m-rho-def}, and let $H^i =H_\iota\left(\rho_i, \frac{r}{a_i}, r \right) \subset \mathbb{S} \times \mathbb{R}$ following Notation~\ref{H-notation}.  Let $K =\prod_{i=1}^3 H^i \subset (\mathbb{S} \times \mathbb{R})^3$.  Then we have $\Psi(K) \subset B_g(M^{\prime}r)$.
\end{theo}

\begin{proof}
  If $(x_i, y_i) \in H^i$, we can write $x_i = \mu_i + \nu_i$ and $y_i = d \nu_i + \xi_i$, where $|\mu_i| \leqslant \rho_i$, $|\nu_i| \leqslant \frac{r}{a_i}$, and $|\xi_i| \leqslant r$.  Lemma~\ref{identity-two-brackets} states that, when $\eta, M^\prime$ are appropriately chosen, we have that $|\mu_i| \leqslant  \rho_i$ implies $e^{\mu_i u_i} \in B_g(M^\prime r)$.  The condition $|\nu_i| \leqslant \frac{r}{a_i}$ implies $e^{\nu_i v_i} \in B_g(r)$, since as before the Maurer--Cartan form of $t \mapsto e^{t v_i}$ is $v_i$, and likewise $e^{\xi_i f_i} \in B_g(r)$.  Hence
  \begin{equation*}
    e^{x_i u_i} e^{y_i f_i} = e^{\mu_i u_i} e^{\nu_i v_i} e^{\xi_i f_i} \in B_g((M^\prime +2)r)
  \end{equation*}
for each $i$, and thus
  \begin{equation*}
 \Psi(x_1, y_1, x_2, y_2, x_3, y_3) = \prod_{i=1}^3 e^{x_i u_i}
 e^{y_i f_i} \in B_g(3(M^{\prime}+2)r)
  \end{equation*}
for all $(x_1, y_1, x_2, y_2, x_3, y_3) \in \prod_{i=1}^3 H^i = K$. This is the desired statement, with $M^{\prime\prime} = 3(M^{\prime}+2)$.
\end{proof}

\begin{cor}\label{lower-cor}
There exist universal constants $\eta, c$ such that for all decoupled left-invariant Riemannian metrics $g$ on $\operatorname{SU}(2) \times \mathbb{R}^3$ and all $r > 0$, we have
\begin{equation}
\mu_0(B_g(r)) \geqslant c \overline{V}_g(r)
\end{equation}
where $\overline{V}_g(r)$ is defined in terms of $r$, $\eta$, and the parameters of $g$, as in Theorem \ref{volume-estimate}.
\end{cor}

\begin{proof}
  We continue to follow the notation of Theorem~\ref{ball-contain-inner}.  Since $K \subset A_\iota$ (see Notation \ref{iota}), we have
  \begin{align*}
 \mu_0(B_g(M^{\prime\prime} r)) &\geqslant \mu_0(\Psi(K)) \\
 &\geqslant \frac{1}{2} |K| \\
 &= \frac{1}{2} \prod_{i=1}^3 \left|H_\iota \left(\rho_i, \frac{r}{a_i},
 r\right)\right| \\
 &\geqslant \frac{1}{2} (c')^3 \overline{V}_g(r)
  \end{align*}
  by Lemma \ref{H-iota-lower}.

  Replacing $r$ by $\frac{r}{M^{\prime\prime}}$, we have shown that $\mu_0(B_g(r)) \geqslant c \overline{V}_g(\frac{r}{M^{\prime\prime}})$ for all $r > 0$.  But since $\overline{V}_g$ is a uniformly doubling function (see Lemma \ref{unif-doubling-abstract} and subsequent remarks), we have $\overline{V}_g(2r) \leqslant D \overline{V}_g(r)$ for some universal constant $\overline{D}$.  Thus, if we choose an integer $p$ large enough that $2^p \geqslant M^{\prime\prime}$, we have $\overline{V}_g(\frac{r}{M^{\prime\prime}}) \geqslant \overline{V}_g(2^{-p} r) \geqslant \overline{D}^{-p} \overline{V}_g(r)$, which yields the conclusion when we take $c = \frac{1}{2} (c')^3 \overline{D}^{-p}$.
\end{proof}

\section{Upper bounds}\label{upper-sec}

For this section, recall that when considering the parameters $(a_1, a_2, a_3, d)$ of a decoupled left-invariant Riemannian metric $g$ on $\mathrm{SU}(2) \times \mathbb{R}^3$, we assume without loss of generality that $a_1 \leqslant a_2 \leqslant a_3$.  The upper bounds divide into two cases, depending on whether $r \leqslant \eta a_2$ or $r \geqslant \eta a_2$. In the latter case, the ball $B_g(r)$ effectively fills up all of the $\operatorname{SU}(2)$ factor and so the estimate becomes much simpler.

\begin{theo}\label{ball-contain-outer}
There exist universal constants $\eta, C$ such that the following holds for every decoupled left-invariant Riemannian metric $g$ on  $\operatorname{SU}(2) \times \mathbb{R}^3$, with its decoupled basis and parameters denoted as before, and every $0 < r \leqslant \eta a_2$.  Define  $m_i, \rho_i$ in terms of $r,a_i,\eta$ as in Notation~\ref{m-rho-def}, and let $H^i = H\left(C \rho_i, \frac{r}{a_i}, r \right) \subset \mathbb{S} \times \mathbb{R}$ following Notation~\ref{H-notation}.  Let $K = \prod_{i=1}^3 H^i \subset (\mathbb{S} \times \mathbb{R})^3$.  Then  we have $B_g(r) \subset \Psi(K)$.
\end{theo}

\begin{proof}
Suppose that $p \in B_g(r)$, so that there exists a smooth curve $\gamma :[0,1] \to G$ from $e$ to $p$ with length $\ell_g[\gamma] < r$.  We may write the Maurer--Cartan form $c_\gamma  : [0,1] \to \mathfrak{g}$  in the decoupled Milnor basis $(v_1, v_2, v_3, f_1, f_2, f_3)$ as
\begin{align*}
 c_\gamma(t) &= \sum_{i=1}^3 \alpha_i(t) v_i + \beta_i(t) f_i \\
 &= \sum_{i=1}^3 \alpha_i(t) u_i + (d \alpha_i(t) + \beta_i(t)) f_i.
\end{align*}
Reparametrizing by constant speed and recalling the orthogonality of the decoupled Milnor basis $(v_i, f_i)$, we can assume that
  \begin{equation*}
 g(c_\gamma(t), c_\gamma(t)) = \sum_{i=1}^3 a_i^2 \alpha_i(t)^2 +
 \beta_i(t)^2 \leqslant r^2, \qquad 0 \leqslant t \leqslant 1
  \end{equation*}
so that in particular, $\alpha_i(t) \leqslant r/a_i$ and $\beta_i(t) \leqslant r$ for all $t$.  Set $\nu_i(t) = \int_0^t \alpha_i(s)\,ds$ and $\xi_i(t) = \int_0^t \beta_i(s)\,ds$.

We now wish to find functions $x_i, y_i : [0,1] \to \mathbb{R}$ satisfying
\begin{equation*}
 \gamma(t) = \Psi(x_1(t), y_1(t), x_2(t), y_2(t), x_3(t), y_3(t)).
\end{equation*}
Take $y_i(t) = d \nu_i(t) + \xi_i(t)$ for $i=1,2,3$; then it is clear that both sides agree in their $\mathbb{R}^3$ components. Next, we will take $(x_1(t), x_2(t), x_3(t))$ to be the solution of
  \eqref{e.MC-ODE}, with $x_i(0)=0$.  To see that a solution exists, note the equation for $x_2(t)$ is
\begin{equation}
x_2^{\prime}(t) = (\cos x_1(t)) \alpha_2(t) - (\sin x_1(t)) \alpha_3(t)
\end{equation}
so that in particular
\begin{equation*}
 |x_2^{\prime}(t)| \leqslant |\alpha_2(t)| + |\alpha_3(t)| \leqslant \frac{r}{a_2} +
 \frac{r}{a_3} \leqslant 2 \frac{r}{a_2}.
\end{equation*}
Choosing for instance $\eta <\pi/6$, integrating yields that $|x_2(t)| \leqslant 2 r / a_2 < 2\eta < \pi/3$ when  it exists, and thus $|\tan x_2(t)| \leqslant |\sec x_2(t)| \leqslant 2$.  Hence  all the coefficients in \eqref{e.MC-ODE} are bounded in absolute  value by $2$, so in fact the solution exists for $0 \leqslant t \leqslant 1$.  Moreover, we have
\begin{align*}
|x_1^{\prime}(t)| &\leqslant |\alpha_1(t)| + 2 |\alpha_2(t)| + 2 |\alpha_3(t)|
\leqslant 5 \frac{r}{a_1},
\\
|x_3^\prime(t)| & \leqslant 2 |\alpha_2(t)| + 2 |\alpha_3(t)| \leqslant 4 \frac{r}{a_2}
\end{align*}
and we integrate to obtain $|x_1(t)| \leqslant 5r/a_1$, $|x_3(t)| \leqslant 4r/a_2$ for all $0 \leqslant t \leqslant 1$.

Now set $\mu_i(t) = x_i(t) - \nu_i(t)$, so that $\mu_i^{\prime}(t) = x_i^{\prime}(t)- \alpha_i(t)$, and thus $(\mu_1(t), \mu_2(t), \mu_3(t))$ satisfies the ODE system
  \begin{equation}\label{mu-ODE}
 \begin{pmatrix}
   \mu_1^{\prime} \\
   \mu_2^{\prime} \\
   \mu_3^{\prime}
 \end{pmatrix} =
 \begin{pmatrix}
    0 & \sin x_1 \tan x_2 & \cos x_1 \tan x_2 \\
 0 & \cos x_{1} - 1 & -\sin x_{1} \\
 0 & \sin x_{1} \sec x_2 & \cos x_{1} \sec x_2 - 1 \\
 \end{pmatrix}
 \begin{pmatrix}
   \alpha_1 \\ \alpha_2 \\ \alpha_3
 \end{pmatrix}.
  \end{equation}
Following Notation~\ref{m-rho-def} with $m_i =\min(r/a_i, \eta)$, so that $m_2 = r/a_2$, $m_3 = r/a_3$, we observe that, for some sufficiently large $C$, we have
\begin{align*}
|\sin x_1| &\leqslant \min(|x_1|, 1) \leqslant \min \left(5 \frac{r}{a_1}, 1\right) \leqslant C m_1 \\
 |\cos x_1 - 1| &\leqslant \min(x_1^2, 2) \leqslant C m_1^2 \\
 |\tan x_2| &\leqslant 2 |x_2| \leqslant C m_2,
\end{align*}
since we noted above that $|x_1| \leqslant Cr/a_1$, $|x_2| \leqslant C r/a_2 = C m_2 \leqslant C \eta$, and can choose $\eta$ small enough to ensure $|\tan x_2| \leqslant 2 |x_2|$.  Finally, we also have
\begin{align*}
 |\cos x_1 \sec x_2 - 1| &= |\cos x_1 (\sec x_2 - 1) + (\cos x_1 -
 1)| \\
 &\leqslant C (m_2^2 + m_1^2) \leqslant 2 C m_1^2,
\end{align*}
where we used $|\cos x_1| \leqslant 1$, $|\sec x_2 - 1| \leqslant C x_2^2 \leqslant C' m_2^2$ for $|x_2| \leqslant \pi/3$, and $|\cos x_1 - 1| \leqslant C m_1^2$ as noted above. So we obtain the differential inequalities
\begin{align*}
|\mu_1^{\prime}(t)| \leqslant C m_1 m_2 |\alpha_2(t)| + C m_2 |\alpha_3(t)|, 
 \\
|\mu_2^{\prime}(t)| \leqslant C m_1^2 |\alpha_2(t)| + C m_1 |\alpha_3(t)|, 
\\
|\mu_3^{\prime}(t)| \leqslant 2 C m_1 |\alpha_2(t)| + 2 C m_1^2 |\alpha_3(t)| .
\end{align*}
Integrating and using that $|\alpha_i(t)| \leqslant r/a_i = m_i$ for $i=2,3$, we see that
\begin{align*}
|\mu_1(1)| \leqslant C m_1 m_2^2 + C m_2 m_3 \leqslant C \rho_1,
\\
|\mu_2(1)| \leqslant C m_1^2 m_2 + C m_1 m_3 \leqslant C \rho_2, 
\\
|\mu_3(1)| \leqslant C m_1 m_2 + C m_1^2 m_3 \leqslant C \rho_3.
\end{align*}
We have thus shown that for each $i$ we can write $x_i(1) = \mu_i(1)+ \nu_i(1)$, $y_i(1) = d \nu_i(1) + \xi_i(1)$, where $|\mu_i(1)| \leqslant C \rho_i$, $|\nu_i(1)| \leqslant r/a_i$, $|\xi_i(1)| \leqslant r$.  That is, we  have $(x_i(1), y_i(1)) \in H^i = H(C \rho_i, r/a_i, r)$, which is to say that $p \in \Psi(K)$.
\end{proof}

To handle the case $r \geqslant \eta a_2$, a much more trivial bound suffices.

\begin{pro}\label{linear-upper}
For every decoupled left-invariant Riemannian metric $g$ on  $\operatorname{SU}(2) \times \mathbb{R}^3$, with its decoupled basis and parameters denoted as before, and every $r > 0$, we have
\begin{equation*}
 B_g(r) \subset \operatorname{SU}(2) \times \prod_{i=1}^3 \left[ - \left(d
   \frac{r}{a_i} + r\right), d
   \frac{r}{a_i} + r\right]
\end{equation*}
and so $\mu_0(B_g(r)) \leqslant 8 \mu_{0, \operatorname{SU}(2)}(\operatorname{SU}(2)) \prod_{i=1}^3 (d \frac{r}{a_i} + r)$, where $\mu_{0, \operatorname{SU}(2)}(\operatorname{SU}(2))$ is the measure of $\operatorname{SU}(2)$ with respect to the $\operatorname{SU}(2)$ factor of the  fixed Haar measure $\mu_0$.
\end{pro}

\begin{proof}
As in the proof of Theorem~\ref{ball-contain-outer} above, write $p = \gamma(1)$, where
  \begin{align*}
 c_\gamma(t)  &= \sum_{i=1}^3 \alpha_i(t) u_i + (d \alpha_i(t) + \beta_i(t)) f_i
  \end{align*}
  with $|\alpha_i(t)| \leqslant r/a_i$, $|\beta_i(t)| \leqslant r$.  Tracking only
  the coefficients of $f_i$, we can write
  \begin{equation*}
\gamma(t) = \gamma_0(t) \prod_{i=1}^3 \exp((d \nu_i(t) + \xi_i(t)) f_i),
\end{equation*}
where $\gamma_0$ is a curve in $\operatorname{SU}(2)$, and as before $\nu_i(t) = \int_0^t \alpha_i(s)\,ds$, $\xi_i(t) = \int_0^t \beta_i(s)\,ds$.  In particular, $|\nu_i(1)| \leqslant r/a_i$ and
  $|\xi_i(1)| \leqslant r$, so the coefficient of $f_i$ is at most $d r/a_i
  + r$ in absolute value.
\end{proof}

\begin{cor}\label{cor-upper}
There exist universal constants $\eta, C$ such that for all $g \in \mathcal{L}_{\operatorname{dec}}(\operatorname{SU}(2) \times \mathbb{R}^3)$, we have
\begin{equation}
 \mu_0(B_g(r)) \leqslant C \overline{V}_g(r),
\end{equation}
where $\overline{V}_g(r)$ is as defined in Theorem~\ref{volume-estimate}.
\end{cor}

\begin{proof}
Taking $C, \eta$ from Theorem~\ref{ball-contain-outer}, we have the result for $r \leqslant \eta a_2$, after adjusting constants as we have done before.  Now suppose that $r \geqslant \eta a_2$.  We now have $\frac{r}{a_1} \geqslant \frac{r}{a_2} \geqslant \eta$, and so in the notation of Theorem~\ref{volume-estimate}, we have $m_1 = m_2 = \eta$, and thus for some small constant $c$ we have $\rho_i \geqslant c$ for $i=1,2,3$, where we could take $c = \eta^3$ when assuming $\eta < 1$. This implies
  \begin{equation*}
 d \rho_i \frac{r}{a_i}
 + \rho_i r + \frac{r^2}{a_i} \geqslant c \left(d\frac{r}{a_i} + r\right)
  \end{equation*}
  where we can simply discard the $\frac{r^2}{a_i}$ term.  Thus, in this case we have
  \begin{equation*}
 \overline{V}_g(r) \geqslant c^3 \prod_{i=1}^3 \left(d \frac{r}{a_i} + r\right)
  \end{equation*}
which when combined with Proposition~\ref{linear-upper} shows $\mu_0(B_g(r)) \leqslant C \overline{V}_g(r)$ for an appropriate choice of $C$.
\end{proof}

\section{Consequences and applications}\label{s.consequences}

While the uniform doubling property for Lie groups is interesting from a purely geometric standpoint, it has significant implications for analysis and probability on such a group.  As was mentioned in Section~\ref{s.intro}, the key fact is that for any unimodular Lie group $G$ with a left-invariant metric $g$, if $(G,g)$ is volume doubling with some constant $D_g$, then it satisfies the scale-invariant Poincar\'e inequality \eqref{poincare-intro} with that same constant $D_g$.  See \cite[Section~5.6.1]{Saloff-CosteBook2002} for a proof. Thus, if $G$ is \emph{uniformly} doubling with constant $D(G)$, then every left-invariant metric $g \in \mathfrak{L}(G)$ satisfies \eqref{poincare-intro} with that same constant; we could say that \eqref{poincare-intro} holds \emph{uniformly} over $g \in \mathfrak{L}(G)$.  As such, a corollary of the results of this paper is that this is the case for any group $G$ that is a quotient of $\mathrm{SU}(2) \times \mathbb{R}^n$ for some $n$, where indeed the constant $D(G)$ depends only on $n$.  

In \cite[Section~8]{EldredgeGordinaSaloff-Coste2018}, we discussed several other functional inequalities on $(G,g)$ which follow from volume doubling via the Poincar\'e inequality, for which it can likewise be shown that each one holds with a constant depending only on $D_g$.  Note we do not claim that the \emph{best} constant is a function of $D_g$ alone.  In a uniformly doubling group $G$, such inequalities therefore hold uniformly over all $g \in \mathfrak{L}(G)$ in the same sense as above.  Here, we simply list some of these statements for the groups studied in this paper, and refer to \cite{EldredgeGordinaSaloff-Coste2018} and references therein for details.

We remark that in \cite{EldredgeGordinaSaloff-Coste2018}, our focus was on \emph{compact} Lie groups, and thus the results for uniformly doubling groups were stated there under the hypothesis that the group is compact.  In the present paper, some of the groups for which we have shown uniform doubling are not compact (e.g. $\operatorname{SU}(2) \times \mathbb{R}^n$).  However, many of the results cited in \cite{EldredgeGordinaSaloff-Coste2018} do in fact hold when the group is merely unimodular, as can be seen by referring to the statements or proofs in the original sources.  The exceptions would be the results dealing with the discrete spectrum of the Laplacian, which typically do not exist when $G$ is not compact, and those relying on the fact that the total volume of a compact group $G$ is finite. For example, such results in \cite[Section~8.7]{EldredgeGordinaSaloff-Coste2018} include ergodicity and  describing the rate of convergence of the heat kernel to the equilibrium in terms of the spectral gap.

We introduce some notation.  Let $(G,g)$ be a unimodular Lie group equipped with a left-invariant Riemannian metric $g$.  In what follows:
\begin{itemize}
\item $\mu$ is the Riemannian volume measure induced by $g$, which is bi-invariant;
\item $d$ is the Riemannian distance, and $B(x,r)$ the corresponding ball of radius $r$ centered at $x$;
\item $\mathrm{diam}$ is the diameter of $G$ with respect to $d$;
\item $V(r) = \mu(B(e,r))$ is the volume growth function;
\item $\mathbf{V} = \mu(G) = V(\infty)$ is the volume of $G$ itself;
\item $\nabla$ is the Riemannian gradient induced by $G$, and we write $|\nabla f|^2$ for $g(\nabla f, \nabla f)$;
\item $\Delta$ is the (positive semidefinite) Laplacian;
\item $P_t = e^{-t \Delta}$ is the heat semigroup;
\item $p_t$ is the heat kernel;
\item $0 = \lambda_0 < \lambda_1 \leqslant \lambda_2 \leqslant \dots$ are the eigenvalues of $\Delta$, counted with multiplicity (in a setting where $\Delta$ has discrete spectrum);
\item $\lambda = \lambda_1$ is the spectral gap, the lowest non-zero eigenvalue of $\Delta$;
\item $\mathfrak{W}(s) = \#\{ i : \lambda_i < s\}$ is the Weyl spectral counting function.
\end{itemize}
Note these all depend implicitly on $g$, but we omit any subscript $g$ for easier reading.

\begin{theo}
  Let $n \geqslant 0$.  There are constants $C(n), c(n)$, etc., varying from line to line but depending only on $n$, such that the following statements hold for every Lie group $G$ of the form $G = (\mathrm{SU}(2) \times \mathbb{R}^n) / H$ for some closed normal subgroup $H \leqslant \mathrm{SU}(2) \times \mathbb{R}^n$, and every left-invariant Riemannian metric $g$ on $G$.
  \begin{enumerate}
  \item The scale-invariant Poincar\'e inequality
      \begin{equation*}
    \int_{B(x,r)} |f - f_{x,r}|^2 \,d\mu \leqslant 2 r^2 C(n) \int_{B(x, 2r)} |\nabla f|^2\,d\mu
  \end{equation*}
      for all $f \in C^\infty_c(G)$, where $f_{x,r} = \fint_{B(x,r)} f\,d\mu$ is the mean of $f$ over $B(x,r)$.
    \item The heat kernel upper bounds
      \begin{equation*}
        p_t(x,y) \leqslant C(n) \frac{\left(1+\frac{d(x,y)^2}{4t}\right)^{\kappa(n)}}{V(\sqrt{t})} \exp\left(-\frac{d(x,y)^2}{4t}\right)
      \end{equation*}
      where the exponent $\kappa(n)$ also depends only on $n$.
    \item The heat kernel time derivative upper bounds
      \begin{equation*}
        |\partial^k_t p_t(x,y)| \leqslant C(n,k) \frac{\left(1+\frac{d(x,y)^2}{4t}\right)^{k+\kappa(n)}}{t^k V(\sqrt{t})} \exp\left(-\frac{d(x,y)^2}{4t}\right).
      \end{equation*}
    \item The heat kernel lower bound
      \begin{equation*}
        p_t(x,y) \geqslant c(n) \frac{1}{V(\sqrt{t})} \exp\left(-A(n) \frac{d(x,y)^2}{t}\right).
      \end{equation*}
    \item       The parabolic Harnack inequality
      \begin{equation*}
        \sup_{Q_-} \{u\} \leqslant H(n) \inf_{Q_+} \{u\}
      \end{equation*}
      for every positive solution $u$ of the heat equation in a heat ball $(s, s+4r^2) \times B(x, 2r)$, where $Q_- = (s+r^2, s+2r^2) \times B(x,r)$ and $Q_+ = (s+3r^2, s+4r^2) \times B(x,r)$.
      \setcounter{keepgoing}{\value{enumi}}
  \end{enumerate}
  Moreover, if $G$ is compact, then the following hold as well.
  \begin{enumerate}
    \setcounter{enumi}{\value{keepgoing}}
  \item The spectral gap upper bound
    \begin{equation*}
      \lambda \leqslant C(n) \mathrm{diam}^{-2}
    \end{equation*}
    which matches the well-known universal lower bound $\lambda \geqslant \frac{\pi^2}{4} \mathrm{diam}^{-2}$ due to Peter Li \cite{LiP1980a}.
  \item The spectral counting function lower bound
    \begin{equation*}
      \mathfrak{W}(s) \geqslant c(n) \frac{\mathbf{V}}{V(s^{-1/2})}
    \end{equation*}
    which matches the universal upper bound $\mathfrak{W}(s) \leqslant C \frac{\mathbf{V}}{V(s^{-1/2})}$ obtained by Judge and Lyons in \cite{JudgeLyons2017}.
\item The $L^2$ heat kernel ergodicity estimates
  \begin{equation*}
    \frac{c(n)}{V(\sqrt{t})} e^{-2 \lambda t} \leqslant \| p_t - \mathbf{V}^{-1}\|_{L^2}^2 \leqslant \frac{C(n)}{V(\sqrt{t})} e^{-2 \lambda t}
  \end{equation*}
\item The $L^1$ heat kernel ergodicity estimate
  \begin{equation*}
    \|p_t - \mathbf{V}^{-1}\|_{L^1} \geqslant \frac{1}{\mathbf{V}} e^{-t \lambda}
  \end{equation*}
\item The heat kernel gradient estimate
  \begin{equation*}
    |\nabla p_t(e,x)| \leqslant C(n) \frac{1}{\sqrt{t} V(\sqrt{t})} \left(1 + \frac{d(e,x)^2}{4t}\right)^{\kappa'(n)} \exp\left(-\frac{d(e,x)^2}{4t}\right)
  \end{equation*}
  where the exponent $\kappa'(n)$ also depends only on $n$.  This implies in particular that $\|\nabla p_t\|_{L^1} \leqslant C'(n) t^{-1/2}$.
  \end{enumerate}
\end{theo}

\def\cprime{$'$}

\end{document}